%% file: glt_190107.tex
\documentclass[11pt]{amsart}
\usepackage{amsthm}
\usepackage{amssymb}
\usepackage{amsmath}

\input xy
\xyoption{all} 

\title{Decidability of the theory of modules over Pr\"{u}fer domains with dense value groups}

\author[]{Lorna Gregory}

\address[L.~Gregory]{Dipartimento di Matematica e Fisica, Universit\`a degli Studi della Campania ``Luigi
Vanvitelli'', Viale Lincoln 5, 81100 Caserta, Italy}

\email{lorna.gregory@gmail.com}

\author[]{Sonia L'Innocente}

\address[S.~L'Innocente]{University of Camerino, School of Science and Technologies,
Division of Mathematics, Via Madonna delle Carceri 9, 62032 Camerino, Italy}

\email{sonia.linnocente@unicam.it}

\author[]{Carlo Toffalori}

\address[C.~Toffalori]{University of Camerino, School of Science and Technologies,
Division of Mathematics, Via Madonna delle Carceri 9, 62032 Camerino, Italy}

\email{carlo.toffalori@unicam.it}

\thanks{The second and third authors thank the Italian GNSAGA-INdAM for its support.}


\input{macros}

\begin{document}

\begin{abstract}
We provide algebraic conditions ensuring the decidability of the theory of modules
over effectively given Pr\"ufer (in particular B\'ezout) domains whose localizations at maximal ideals have dense value groups. For B\'ezout domains, these conditions are also necessary.
\end{abstract}

\maketitle

\pagestyle{plain}

\vspace{3mm}

\section{Introduction}\label{intro}

This paper contributes to a body of work characterizing when the theory of modules of a Pr\"{u}fer domain is decidable. This direction of research was initiated by Puninskaya, Puninski and the third author in \cite{PPT}, where it was shown that all effectively given valuation domains with dense Archimedean value group have decidable theory of modules. Confirming a conjecture in \cite{PPT}, the first author proved, in \cite{Gre15}, that the theory of modules of an effectively given valuation domain $V$ is decidable if and only if the set of pairs $(a,b)\in V^2$ such that $a\in\rad(bV)$ is recursive. This work was picked up in \cite{GLPT}, where a complete characterization of effectively given B\'ezout domains $R$ with infinite residue fields and decidable theory of modules was given in terms of the recursivity of a certain subset, $\text{DPR}(R)$, of $R^4$, generalizing the prime radical relation (see later in this introduction for a definition of this set). Analogous sufficient conditions were given for the theory of modules of an effectively given Pr\"{u}fer domain with infinite residue fields to be decidable.

In this article we characterize effectively given B\'ezout domains $R$ with decidable theory of modules under the assumption that the value groups of all localizations of $R$ at maximal ideals are dense. We also give a sufficient condition for the theory of modules of a Pr\"{u}fer domain to be decidable under the same assumption.

Thanks to the Baur-Monk theorem, if $R$ is a recursive ring then the theory of $R$-modules is decidable if and only if there exists an algorithm which, given $\phi_1/\psi_1,\ldots,\phi_h/\psi_h$ pairs of pp-formulae and intervals $[n_1,m_1],\ldots, [n_h,m_h]\subseteq \N\cup\{\infty\}$, answers whether there exists an $R$-module $M$ such that, for all $1\leq i\leq h$, $| \, \phi_i (M) / \psi_i(M) \, |\in [n_i,m_i]$. A standard argument means we may assume that $M$ is a finite direct sum of indecomposable pure injective $R$-modules.

Incidentally, the assumption {\sl \lq \lq $R$ recursive"}, as well as the stronger one {\sl \lq \lq $R$
is effectively given"}, are necessary to guarantee that the decidability problem of the theory of
$R$-modules makes sense. We will recall both of them in $\S$ 2.

Modern proofs of decidability for theories of modules roughly split into two steps. The first step of such a proof gives an algorithm which decides whether one Ziegler basic open set is contained in a finite union of other Ziegler basic open sets. Equivalently, it gives an algorithm, as in the previous paragraph, but where the intervals $[n_i,m_i]$ are either $[1,1]$ or $[2,\infty]$. When the Baur-Monk invariants $|\phi(M)/\psi(M)|$ of all $R$-modules $M$ are either infinite or $1$ then this is enough to show that the theory of $R$-modules is decidable.

However, when there exists $R$-modules with finite Baur-Monk invariants different from $1$, more work is required. This, the second step, usually amounts to a fine detailed analysis of the indecomposable pure injective $R$-modules $N$ such that there exists a pair of pp-formulae $\phi/\psi$ with $|\phi(N)/\psi(N)|$ finite but not equal to $1$. This analysis is then used to reduce to the case of the first step.

For B\'ezout and Pr\"ufer domains, the first step was dealt with in \cite{GLPT}. Note that, if $R$ is a Pr\"{u}fer domain with all residue fields infinite then for all $R$-modules $M$ and pairs of pp-formulae $\phi/\psi$, $| \, \phi(M) / \psi(M) \, |$ is either $1$ or infinite.

On the other hand it was shown, in \cite{PPT}, that if $V$ is valuation domain with finite residue field and dense value group then for each pp-pair $\phi/\psi$ there are at most finitely many indecomposable pure injective $V$-modules $N$ with $| \, \phi(N)/\psi(N) \, |$ finite but not equal to $1$. This makes the combinatorial problem of dealing with finite invariants sentences somewhat easier. In turn, since every indecomposable pure injective over a Pr\"{u}fer domain $R$ is the restriction of an indecomposable pure injective over some localization of $R$ at a maximal ideal, our assumption that $R_{\mfrak{m}}$ has dense value group for all maximal ideals $\mfrak{m}$, makes it easier to deal with finite invariants sentences in this case too.

Before explaining the content of this article in more detail, we fix some notation. For $R$ a ring, let $L_R$ denote the language of $R$-modules, and $T_R$ the theory of $R$-modules. Let $\mathbb{N}$ be the set of positive integers and $\mathbb{N}_0$ the set of the non-negative integers; $\mathbb{P}$ is the set of prime numbers (in $\mathbb{N}$). For $k_1, \ldots, k_s \in \mathbb{N}$,  $\text{\Span}_{\mathbb{N}_0} \{ k_1, \ldots, k_s \}$
denotes the set of linear combinations of $k_1, \ldots, k_s$ with coefficients in $\mathbb{N}_0$.

Our characterization of  B\'ezout domains with decidable theory of modules is based on two key sets.
\begin{itemize}
\item The first is $\text{DPR}(R)$, introduced in \cite[$\S$6]{GLPT}, that is, the set of 4-tuples $(a,b,c,d)$ in $R^4$ such that, for every choice of prime ideals $\mfrak{p}$, $\mfrak{q}$ of $R$ with $\mfrak{p} + \mfrak{q} \neq R$, either $a \in \mfrak{p}$, $b \notin \mfrak{p}$, $c \in \mfrak{q}$ or $d \notin \mfrak{q}$. We call $\text{DPR}(R)$ the {\sl double prime radical relation} because of its similarity to the {\sl prime radical relation} in $R$, $a \in \rad (bR)$ with $a, b \in R$. Note that, $a \in \rad (bR)$ if and only if, for every proper prime ideal $\mathfrak{p}$ of $R$, if $b \in \mathfrak{p}$, then $a \in \mathfrak{p}$, that is, either $a \in \mathfrak{p}$ or $b \notin \mathfrak{p}$. Hence $a \in \rad (bR)$ if and only if $(a, b, a, b) \in \text{DPR}(R)$.
\item The second set is inspired by the characterization of the (effectively given) commutative regular rings with decidable theory of modules given by Point and Prest in \cite{PP}. For this reason we denote it $\text{PP} (R)$. The set $\text{PP}(R)$ consists of the 4-tuples $(p,n,c,d) \in \mathbb{P} \times \N \times R^2$ such that there exist
positive integers $s, k_1,\ldots, k_s$ and maximal ideals $\mfrak{m}_1,\ldots,\mfrak{m}_s$ of $R$
for which $n\in\text{Span}_{\N_0}\{k_1,\ldots,k_s\}$ and for all $i = 1, \ldots, s$,
\begin{enumerate}
\item $|R/\mfrak{m}_i|=p^{k_i}$,
\item $c\in \mfrak{m}_i$,
\item $d\notin \mfrak{m}_i$.
\end{enumerate}
\end{itemize}

By definition, see $\S$\ref{prel}, the elements of an effectively given Pr\"ufer domain $R$ come equipped with an enumeration, so it makes sense to say that a subset of $R^n$ for some $n\in\N$ or $\bigcup_{n \in\N}R^n$ is recursive.

The main result for B\'ezout domains of this article is the following.

\medskip

\noindent
{\bfseries Theorem 6.2.} \textit{Let $R$ be an effectively given B\'ezout domain such that each localization of $R$ at a maximal ideal has dense value group. Then $T_R$ is decidable if and only if both $\text{DPR} (R)$ and $\text{PP}(R)$ are recursive.}

\medskip


For Pr\"{u}fer domains the situation is more complicated. This time we introduce, for every $l \in \N$,
two sets $\text{DPR}_l (R) \subseteq R^{2l+2}$ and $\text{PP}_l (R) \subseteq \P \times \N \times
R^{l+1}$.
\begin{itemize}
\item $\text{DPR}_l (R)$ is the set of $(2l+2)$-tuples $(a, b_1, \ldots, b_l, c, d_1, \ldots, d_l) \in
R^{2l+2}$ such that, for every choice of prime ideals $\mfrak{p},\mfrak{q}$ of $R$ with
$\mfrak{p} + \mfrak{q} \neq R$, either $a \in \mfrak{p}$, $c \in \mfrak{q}$, $b_j \notin \mfrak{p}$
or $d_j \notin \mfrak{q}$ for some $1\leq j \leq l$. Note that $\text{DPR}_1 (R) = \text{DPR} (R)$.

We put $\text{DPR}^\star (R) = \bigcup_{l \in \N}
\text{DPR}_l (R)$. Note that, when $R$ is B\'ezout and $\mfrak{p}$ is a proper prime ideal of $R$,
then $b_j \notin \mfrak{p}$ for some $1\leq j \leq l$ if and only if the greatest common divisor of
$b_1, \ldots, b_l$ is not in $\mfrak{p}$. Thus, for $R$ an effectively given B\'ezout domain,
$\text{DPR}^\star (R)$ is recursive if and only if $\text{DPR} (R)$ is.

Recall that the
sets $\text{DPR}_l (R)$ were already considered in \cite[$\S$7]{GLPT} in order to partially extend the main decidability theorem there from B\'ezout to Pr\"ufer domains with infinite residue fields.
As observed in \cite[Theorem 7.1]{GLPT}, an algorithm deciding membership of the $\text{DPR}_l (R)$ uniformly
in $l$ ensures, under the infinite residue fields
assumption, that $T_R$ itself is recursive\footnote{There is an omission in the statement of the published version of \cite[7.1]{GLPT}; the algorithms deciding membership of the various $\text{DPR}_l(R)$ need to be uniform in $l$ and this is also what is explicitly used in the proof.}. The existence of an algorithm deciding membership of the $\text{DPR}_l (R)$ uniformly
in $l$ is equivalent to $\text{DPR}^\star (R)$ being recursive. 


\item For every $l \in \N$, let $\text{PP}_l (R)$ consist of the tuples
$(p, n, c_1, \ldots, c_l, d) \in \P \times \N \times R^{l+1}$ such that there exist
positive integers $s, k_1,\ldots k_s$ and maximal ideals $\mfrak{m}_1,\ldots,\mfrak{m}_s$ of $R$
for which $n\in\text{Span}_{\N_0}\{k_1,\ldots,k_s\}$ and for all $i = 1, \ldots, s$,
\begin{enumerate}
\item $|R/\mfrak{m}_i|=p^{k_i}$,
\item $c_j \in \mfrak{m}_i$ for $1 \leq j \leq l$,
\item $d\notin \mfrak{m}_i$.
\end{enumerate}
Moreover put $\text{PP}^\star (R) = \bigcup_{l \in \N} \text{PP}_l (R)$. Once again, if $R$ is B\'ezout and
effectively given, then $\text{PP}^\star (R)$ is recursive if and only if $\text{PP} (R) = \text{PP}_1 (R)$
is. This is again because, if $\mfrak{m}$ is a maximal ideal of $R$, then $c_1, \ldots, c_l \in \mfrak{m}$ if
and only if the greatest common divisor of $c_1, \ldots, c_l$ is in $\mfrak{m}$.
\end{itemize}

As for B\'ezout domains, \cite[6.4]{GLPT}, if $R$ is an effectively given Pr\"ufer domain with decidable theory of modules then $\text{DPR}(R)$ is recursive. However, for Pr\"ufer domains, we don't know if $T_R$ decidable implies $DPR_{l}(R)$ is recursive for any $l\geq 2$. In particular, we don't know if $T_R$ decidable implies that $\text{DPR}^\star(R)$ is recursive.

\smallskip

\noindent
{\bfseries Theorem 3.2.} \textit{Let $R$ be an effectively given Pr\"ufer domain. If $T_R$ is decidable, then $\text{PP}^\star (R)$ is recursive.}

\smallskip

The main result for Pr\"ufer domains of this article is the following.

\smallskip

\noindent
{\bfseries Theorem 6.1.} \textit{Let $R$ be an effectively given Pr\"ufer domain such that each localization of $R$ at a maximal ideal has dense value group. If both $\text{DPR}^\star (R)$ and $\text{PP}^\star (R)$ are recursive, then $T_R$ is decidable.}

\smallskip

Section \ref{prel} provides some basic information about model theory of modules over Pr\"ufer and B\'ezout domains. In $\S$\ref{recur} we prove Theorem \ref{key1bis}. Sections \ref{order} and  \ref{prep} prepare for the proof of the main theorems. The proofs of \ref{key2bis} and \ref{key33} are contained in $\S$\ref{main}.

We assume some familiarity with basic model theory of modules, as illustrated in Prest's fundamental books \cite{Preb1}, \cite{Preb2} and in the capital paper \cite{Zi}. Pr\"ufer domains and in particular B\'ezout domains are treated in \cite{F-S} and \cite{Gi}. Other recent papers dealing with decidability of modules over B\'ezout domains or related questions include \cite{PT14}, \cite{LPT} and \cite{LPPT}, while \cite{PT15} provides a general treatment of the model theory of modules over B\'ezout domains.

Domains are assumed to be commutative with unity, and modules right unital.

\section{Preliminaries}\label{prel}

Recall that a domain $R$ is \emph{Pr\"ufer} if all its localizations at maximal ideals, and consequently at non-zero prime ideals, are valuation domains, and \emph{B\'ezout} if every 2-generated ideal (and consequently every finitely generated ideal) is principal. Thus $R$ is B\'ezout if and only if the so called B\'ezout identity holds: for every $0 \neq a, b\in R$ there are $c, u, v, g, h \in R$ such that $au+ bv= c$ and $cg=a$, $ch= b$ hold. Then $c$ is called a \emph{greatest common divisor} of $a$ and $b$ and is unique up to a multiplicative unit. B\'ezout domains are Pr\"ufer.

We will make frequent use of the following result of Tuganbaev \cite{Tu}.

\begin{fact}\label{Tuganbaev}
If $R$ is a Pr\"ufer domain then for all $a,b\in R$ there exist $\alpha,r,s\in R$ such that $b\alpha=as$ and $a(\alpha-1)=br$.
\end{fact}

When proving decidability results about $T_R$, for $R$ a Pr\"ufer domain, we will work under the hypothesis that $R$ is effectively given. A Pr\"ufer domain $R$ is  \emph{effectively given} if it is countable and its elements can be listed as $a_0=0, a_1= 1, a_2, \dots$ (possibly with repetitions) so that suitable algorithms effectively execute the following, when $m, n$ range over natural numbers.
\begin{enumerate}
\item Deciding whether $a_m= a_n$ or not.
\item Producing $a_m+ a_n$ and $a_m\cdot a_n$, or rather indices of these elements in the list.
\item Establishing whether $a_m$ divides $a_n$.
\end{enumerate}
This is a natural assumption to ensure that the decidability problem of $T_R$ makes sense. In particular, if $(1)$ and $(2)$ hold then $T_R$ is recursively axiomatizable. Moreover, in order for the theory of $R$-modules to be decidable, conditions $(1)-(3)$ must hold. For instance, $a_m$ divides $a_n$ if and only if the sentence $\forall x (xa_m=0\rightarrow xa_n=0)$ is in $T_R$.

\begin{fact}\label{pp1}
Every pp-$1$-formula over a Pr\"ufer domain is equivalent to a finite sum of
formulae $\exists y (ya =x \land yb = 0)$, with $a, b \in R$, and also to a finite conjunction of formulae $c | xd$, with $c, d \in R$ (see \cite[2.2]{PT15}). Over a B\'ezout domain $R$ a stronger result holds, since every pp-1-formula is equivalent to a finite sum of formulae of the form $a|x \, \land xb=0$, $a, b \in R$, and also to a finite conjunction of formulae $c|x \, + \, xd =0$ with $c, d \in R$ (see for instance \cite[2.3]{PT15}).
\end{fact}

We will denote a pp-1-pair, that is, an ordered pair of pp-1-formulae $\phi, \psi$, by $\phi / \psi$. Of course we will be mainly interested in the cases when $\phi$, $\psi$ have one of the forms in Fact \ref{pp1}.

For $\phi$, $\psi$ pp-1-formulae and $t$ a positive integer, let $| \phi \, / \,
\psi| \geq t$ be the sentence of $L_R$ saying that the index of the pp-subgroup defined by $\phi \land \psi$ inside that defined by $\phi$ is  at least $t$. We call such sentences {\sl invariants sentences} and for $N$ an $R$-module, we call the values of $| \phi(N) \, : \, \phi(N) \cap \psi(N) |$ (either finite or $\infty$) the {\sl Baur-Monk invariants} of $N$. In the following we will abbreviate $| \phi(N) \, : \, \phi(N) \cap \psi(N) |$ by $| \phi / \psi (N) |$. We will say that an $R$-module $N$ {\sl opens} a pp-1-pair $\phi / \psi$ when $|\phi / \psi (N) | > 1$. The already mentioned Baur-Monk theorem, \cite[2.15]{Preb1} asserts that, relative to $T_R$, every sentence in $L_R$ is equivalent to a boolean combination of invariants sentences. Moreover, \cite[2.18]{Preb1}, an $R$-module $N$ is determined up to elementary equivalence by its Baur-Monk invariants.

We will need the following well-known argument: Let $\sigma$ be a boolean combination of invariants sentences. Since, \cite[4.36]{Preb1}, every $R$-module is elementary equivalent to a direct sum of indecomposable pure injective modules, $\sigma$ is satisfied by some $R$-module if and only if $\sigma$ is satisfied by a direct sum of indecomposable pure injective $R$-modules. Essentially because solution sets of pp-formulae commute with direct sums, if $\sigma$ is satisfied by a direct sum of indecomposable pure injectives then $\sigma$ is satisfied by a finite direct sum of indecomposable pure injectives.

Thus, in order to prove that an effectively given Pr\"ufer domain $R$ has decidable theory of modules, it is enough to show that there is an algorithm which, given a conjunction of invariants sentences and negations of invariants sentences $\sigma$, answers whether there exists a finite direct sum of indecomposable pure injective $R$-modules satisfying $\sigma$.

\begin{fact}\label{ppuni}
If $R$ is a Pr\"{u}fer domain and $N$ is an indecomposable pure injective $R$-module then $N$ is \emph{pp-uniserial}, i.e. its lattice of pp-definable subgroups is totally ordered.
\end{fact}

This follows from \cite[3.3]{Pun}, recalling that the lattice of pp-1-formulae over a Pr\"ufer domain is distributive \cite[3.1]{EkHe}.
Recall also that a module $N$ is said to be \emph{uniserial} if the lattice of all its submodules is totally ordered.

For every $R$-module $N$ we put
\[\Ass \, N:=\{r\in R \, \mid \, \text{ there exists } m\in N\backslash\{0\} \text{ with } mr=0\}\]
and
\[\Div \, N:= \{r \in R \, \mid \, r  \not | \, m \; \text{ for some } m\in N\} . \]
If $N$ is an indecomposable pure injective module over a Pr\"ufer domain $R$, then
$\Ass \, N$ and $\Div \, N$ and their union $\Ass \, N \cup \Div \ N$ are (proper) prime ideals of $R$
(see \cite[Lemma 2.7]{GLPT}).

The {\sl Ziegler spectrum} of a ring $R$, $\Zg_R$, is a topological space whose points are (isomorphism classes of) indecomposable pure injective $R$-modules, and whose topology is given by basic open sets of the form $(\phi / \psi)$ where $\phi$ and $\psi$ range over pp-1-formulae of $L_R$. Recall that an open set $(\phi / \psi)$ consists of the $R$-modules $N$ in $\Zg_R$ such that $\phi(N)$ strictly includes its intersection with $\psi(N)$.

For any commutative ring $R$, if $N\in\Zg_{R_{\mfrak{m}}}$ for some maximal ideal $\mfrak{m}$ of $R$ then $N$ restricted to $R$ is an indecomposable pure injective $R$-module. This gives, \cite[5.53]{Preb2}, a homeomorphic embedding of $\Zg_{R_{\mfrak{m}}}$ into $\Zg_R$ as a closed subset.  Moreover, for all $N\in\Zg_R$, there exists a maximal ideal $\mfrak{m}$ of $R$ such that $N$ is the restriction of an indecomposable pure injective $R_{\mfrak{m}}$-module (see for instance \cite[6.4]{Gre13}).

Let us also recall the correspondence, over a valuation domain $V$, between ordered pairs of proper ideals of $V$ and indecomposable pp-1-types over $V$. The indecomposable pp-1-type associated to an ordered pair $(I,J)$ of ideals is just the unique complete pp-1-type $p = p(I, J)$ such that, for all $r \in V$,
\begin{itemize}
\item $xr = 0 \in p$ if and only if $r \in I$ and
\item $r \mid x \in p$ if and only if $r \notin J$,
\end{itemize}
see \cite[3.4]{EkHe} and \cite{GLPT} for more details. Via these indecomposable pp-1-types, pairs
of ideals also correspond to indecomposable pure injective $V$-modules. For every pair $(I, J)$,
let $\text{PE}(I, J)$ denote the indecomposable pure injective $V$-module associated to it
as the pure injective hull of the indecomposable pp-1-type $p(I, J)$. So PE means pure injective hull. The equivalence relation
connecting two pairs $(I, J)$ and $(I', J')$ if and only if $\text{PE}(I,J)$ and $\text{PE}(I',J')$ are isomorphic is well characterized, see again \cite[3.4]{EkHe}.

Before finishing this section we mention a slight peculiarity. The very attentive reader of \cite{GLPT} and this article might be puzzled by the fact that we never actually need to use condition $(3)$ of the definition of an effectively given Pr\"ufer domain. Combined with \ref{Tuganbaev}, the following remark implies that if $R$ is a recursive Pr\"{u}fer domain (i.e. $(1)$ and $(2)$ in the definition of ``effectively given'' hold) and $\text{DPR}(R)$ is recursive then condition $(3)$ in the definition of an effectively given Pr\"{u}fer domain must also hold.

\begin{remark}\label{rem}
Let $a,b\in R\backslash\{0\}$ and $\alpha,s,r\in R$ be such that $b\alpha=as$ and $a(\alpha-1)=br$. Then $b\in aR$ if and only if $(1,\alpha,1,r)\in\text{DPR}(R)$.
\end{remark}
\begin{proof}
For any domain $R$, $b\in aR$ if and only if $b\in aR_{\mfrak{p}}$ for all proper prime ideals $\mfrak{p}$. By \cite[5.5]{GLPT}, $b\in aR_{\mfrak{p}}$ if and only if $\alpha\notin \mfrak{p}$ or $r\notin \mfrak{p}$. So $b\in aR$ if and only if for all prime ideals $\mfrak{p}$, $\alpha\notin \mfrak{p}$ or $r\notin\mfrak{p}$.

From the definition of $\text{DPR}(R)$, it follows that $(1,\alpha,1,r)\in \text{DPR}(R)$ if and only if for all primes $\mfrak{p},\mfrak{q}$ such that $\mfrak{p}+\mfrak{q}\neq R$, $\alpha\notin \mfrak{p}$ or $r\notin \mfrak{q}$.

If $(1,\alpha,1,r)\in \text{DPR}(R)$ then, setting $\mfrak{p}=\mfrak{q}$ in the definition of $\text{DPR}(R)$, it follows that, for all proper primes $\mfrak{p}$, $\alpha\notin \mfrak{p}$ or $r\notin \mfrak{p}$.

Conversely, suppose that for all proper prime ideals $\mfrak{p}'$, $\alpha\notin \mfrak{p}'$ or $r\notin\mfrak{p}'$. Suppose that $\mfrak{p},\mfrak{q}$ are prime ideals and $\mfrak{p}+\mfrak{q}\neq R$. Since $R$ is a Pr\"{u}fer domain, $\mfrak{p}+\mfrak{q}\neq R$ implies $\mfrak{p}\subseteq \mfrak{q}$ or $\mfrak{q}\subseteq \mfrak{p}$. Let $\mfrak{p}'=\mfrak{p}\cup\mfrak{q}$. Then $\alpha\notin \mfrak{p}'$ implies $\alpha\notin \mfrak{p}$ and $r\notin\mfrak{p}'$ implies $r\notin \mfrak{q}$. Therefore $(1,\alpha,1,r)\in \text{DPR}(R)$.
\end{proof}

\section{Recursive sets}\label{recur}

Throughout this section $R$ will be a Pr\"{u}fer domain. However, we will not require that the localizations of $R$ at maximal ideals have dense value groups.

\begin{lemma}\label{nolab}
All non-zero finite modules over a Pr\"{u}fer domain $R$ are of the form $\prod_{i=1}^h R/\mfrak{m}_i^{\lambda_i}$ where $\mfrak{m}_i$ is a maximal ideal and $\lambda_i\in\N$ for every $1\leq i\leq h$. If for each maximal ideal $\mfrak{m}$, $R_\mfrak{m}$ has dense value group then all non-zero finite modules are of the form $\prod_{i=1}^h R/\mfrak{m}_i$ where each $\mfrak{m}_i$ is a maximal ideal.
\end{lemma}
\begin{proof}
Any finite module may be written as a direct sum of indecomposable finite modules, and finite modules are pure injective. Since $R$ is commutative, every indecomposable pure injective is the restriction of a module over $R_\mfrak{m}$ for some maximal ideal $\mfrak{m}$. Since $R$ is Pr\"{u}fer, $R_\mfrak{m}$ is a valuation domain.  If $M$ is a finite module over $R_\mfrak{m}$ then $M$ is a module over a finite quotient of $R_\mfrak{m}$. Thus $M$ is either a module over $R_\mfrak{m}/\pi^nR_\mfrak{m}$ where $\pi$ generates $\mfrak{m}R_\mfrak{m}$ and $R_\mfrak{m}/\mfrak{m}R_\mfrak{m}$ is finite, or $M$ is a module over $R/\mfrak{m}$ where $\mfrak{m}$ is the maximal ideal of $R$ and $R/\mfrak{m}$ is finite.

In order to get the desired result we now just need to note that if $\pi$ generates the maximal ideal of $R_\mfrak{m}$ then $R_\mfrak{m}/\pi^\lambda R_\mfrak{m}$ is isomorphic to the $R$-module $R/\mfrak{m}^\lambda$ for every positive integer $\lambda$.

Finally, when all the value groups are dense, $\mfrak{m} = \mfrak{m}^2$ for every $\mfrak{m}$.
\end{proof}

\begin{theorem}\label{key1bis}
Let $R$ be an effectively given Pr\"ufer domain. If $T_R$ is decidable, then $\text{PP}^\star (R)$ is recursive.
\end{theorem}

\begin{proof}

We claim that, for all $p \in \P$, $n \in \N$, $l \in \N$, $(c_1, \ldots, c_l) \in R^l$ and $d \in R$, $(p,n, c_1, \ldots, c_l, d)$ is in $\text{PP}^\star (R)$ (so in $\text{PP}_l (R)$)  if and only if there is an $R$-module $M$ such that $|M|=p^n$, for each $1 \leq j \leq l$, $c_j\in \ann_RM$ (the annihilator of $M$ over $R$) and for all $m\in M$, $md=0$ implies $m=0$. Note that the condition that $M$ has to satisfy can be expressed as a first order sentence of $L_R$ in terms of $p, n, d$ and the $c_j$.
Moreover there is an effective procedure which, given a tuple $\pi = (p, n, c_1, \ldots, c_l, d)$ with $(c_1, \ldots, c_l)$ of arbitrary length $l$, produces this sentence $\theta_\pi$.
Since $T_R$ is decidable, the set of all sentences of $L_R$ true in at least one $R$-module is recursive. Applying that to the tuples $\pi$ and the corresponding sentences $\theta_\pi$, we obtain an algorithm deciding, for any given $\pi$, whether there exists an $R$-module $M$ satisfying $\theta_\pi$ as required.
Hence $\text{PP}^\star (R)$ is recursive, provided that we prove our claim. Then let us do that.

First suppose that $(p,n,c_1, \ldots, c_l, d)\in\text{PP}_l (R)$. Let $\mfrak{m}_1,\ldots,\mfrak{m}_s$ and $k_1,\ldots$, $k_s$ be as asked in the definition of $\text{PP}_l (R)$, and let $\lambda_1,\ldots,\lambda_s\in\N_0$ satisfy $\sum_{i=1}^s\lambda_ik_i=n$. Put $M=\prod_{i=1}^s(R/\mfrak{m}_i)^{\lambda_i}$. So $|M|=p^{\sum_{i=1}^s\lambda_ik_i}=p^n$. Since for $1 \leq j \leq l$, $c_j \in \mfrak{m}_i$ for all $1\leq i\leq s$, $c_j \in \ann_R M$. Since $d\notin \mfrak{m}_i$ for $1\leq i\leq s$, if $m\in M$ and $md=0$ then $m=0$.

Now suppose that there exists an $R$-module $M$ such that $|M|=p^n$, $c_j \in \ann_RM$ for all $1\leq j\leq l$ and the only element of $M$ annihilated by $d$ is $0$. By Lemma \ref{nolab}, we can assume that $M$ is of the form $\prod_{i=1}^sR/\mfrak{m}_i^{\lambda_i}$ for some suitable maximal ideals $\mfrak{m}_1, \ldots, \mfrak{m}_s$ and positive integers $\lambda_1, \ldots, \lambda_s$. So, for each $j$, $c_j \in \ann_RM$ implies that $c_j \in \mfrak{m}_i$ for $1\leq i\leq s$. We may assume that if $\mfrak{m}_i^2=\mfrak{m}_i$ then $\lambda_i=1$. If $\lambda_i=1$ then $d\notin \mfrak{m}_i$ since $d\in \mfrak{m}_i$ implies $(1+\mfrak{m}_i)d=0$. Suppose $\lambda_i>1$ . Take $m\in \mfrak{m}_i^{\lambda_i-1}\backslash \mfrak{m}_i^{\lambda_i}$. Then $d\in\mfrak{m}_i$ implies $(m+\mfrak{m}_i^{\lambda_i})d=0$. So $d\notin \mfrak{m}_i$. Thus $d\notin\mfrak{m}_i$ for $1\leq i\leq s$.

Now $\vert R/\mfrak{m}_i^{\lambda_i}\vert =|R/\mfrak{m}_i|^{\lambda_i}$. So $p^n=|M|=\prod_{i=1}^s |R/\mfrak{m}_i|^{\lambda_i}$. Hence $|R/\mfrak{m}_i|=p^{k_i}$ for some $k_i\in\N$. It follows $\vert M\vert=p^{\sum_{i=1}^s\lambda_ik_i}$. Thus $n=\sum_{i=1}^s\lambda_ik_i$. Hence $(p,n,
(c_1, \ldots, c_l), d) \in \text{PP}_l (R)$.
\end{proof}

\section{Reducing to divisibility and torsion conditions}\label{order}

When $R$ is an effectively given ring, in order for the theory of $R$-modules to be decidable it is enough that there is an algorithm which, given a sentence of the following form
\[\bigwedge_{i=1}^t \vert\phi_{1,i} / \psi_{1,i} \vert= H_i \wedge \bigwedge_{j=1}^u \vert \phi_{2,j} / \psi_{2,j} \vert = 1 \wedge \bigwedge_{k=1}^s \vert \phi_{3,k} / \psi_{3,k} \vert \geq E_k \tag{$\star$}\]
(where, for every $i = 1, \ldots, t$, $j = 1, \ldots u$ and $k = 1, \ldots s$, $\phi_{1,i}$, $\psi_{1,i}$,
$\phi_{2,j}$, $\psi_{2,j}$ and $\phi_{3,k}$, $\psi_{3,k}$ are pp-1-formulae and $H_i$, $E_k$ are integers $\geq 2$)
answers whether there exists an $R$-module satisfying it. Moreover, this module can be assumed to be a finite direct sum of indecomposable pure injectives.



As also seen in Facts \ref{pp1} and \ref{ppuni}, indecomposable pure injective modules $N$ over valuation domains (and more generally Pr\"ufer domains) are pp-uniserial, and every pp-$1$-formula over a valuation domain $R$ is equivalent to a finite sum of formulae of the form $a|x \, \land xb=0$. Hence, in order to calculate the size of the Baur-Monk invariants of $N$, in particular of those occurring in $(\star)$, it seems enough to handle the problem for pp-pairs $\phi/\psi$ where $\phi,\psi$ are of the form $a|x$ and $xb = 0$ with $a, b \in R$. Since every indecomposable pure injective module over a Pr\"ufer domain $R$ is the restriction of an indecomposable pure injective module over $R_\mfrak{m}$ for some maximal ideal $\mfrak{m}$, this argument transfers to Pr\"ufer domains. This motivates the following result which this section is dedicated to proving.

\begin{theorem}\label{reduction}
Let $R$ be an effectively given Pr\"ufer domain. Suppose that there is an algorithm which, given a sentence
\[\bigwedge_{i=1}^t \vert\phi_{1,i} / \psi_{1,i} \vert = H_i \wedge \bigwedge_{j=1}^u \vert \phi_{2,j} / \psi_{2,j} \vert = 1 \wedge \bigwedge_{k=1}^s \vert \phi_{3,k} / \psi_{3,k} \vert \geq E_k \]
where for $i = 1, \ldots, t$, $j = 1, \ldots, u$ and $k = 1, \ldots, s$, $H_i$, $E_k$ are integers $\geq 2$
and the pp-pairs $\phi_{1,i} \, / \, \psi_{1,i}$, $\phi_{2,j} \, / \, \psi_{2,j}$, $\phi_{3,k} \, / \, \psi_{3,k}$ are of the form $xb = 0 \, / \, c |x$ or $x = x \,  / \, xd = 0$ with $b, c, d \in R$, answers whether there exists an $R$-module satisfying this sentence. Then the theory $T_R$ of $R$-modules is decidable.
\end{theorem}

Before starting the proof, we need some preparatory work.

Let $\Sigma$ be a finite non-empty set of pp-1-formulae. Note that, for every  $R$-module $M$, logical implication (with respect to the theory of $M$) determines a quasi-order on $\Sigma$, which becomes a partial order in the quotient set of $\Sigma$ with respect to the logical equivalence relation (again with respect to the theory of $M$). Both the original quasi-order and the quotient order are
total if $M$ is an indecomposable pure injective module and $R$ is Pr\"ufer (by Fact \ref{ppuni}). With this is mind, let us consider all the possible total quasi-orderings on $\Sigma$ and the corresponding total orderings. To avoid excessively heavy notation, we will identify each total quasi-order on $\Sigma$ with the corresponding total order, and we will denote by $\Gamma(\Sigma)$ the set of these (quasi-)orders. For $L \in \Gamma (\Sigma)$ (with its relation $\leq_L$) and for $\phi,\psi \in \Sigma$, we write
\begin{itemize}
\item $\phi=_L\psi$ to mean that according to $L$, $\phi$ and $\psi$ are equal, that is,
$\phi \geq_L \phi$ and $\psi \geq_L \phi$,
\item $\phi>_L\psi$ to mean that according to $L$, $\phi$ is strictly greater than $\psi$ (so
$\phi \geq_L \psi$ holds, but $\psi \geq_L \phi$ does not).
\end{itemize}
As an example, if $\Sigma:=\{\phi,\psi\}$ then there are $3$ total quasi-orderings, and indeed
3 different related total orders  on $\Sigma$ i.e. those with $\phi=_L\psi$, $\phi>_L\psi$ and
$\psi>_L\phi$ respectively.

For each $L\in\Gamma(\Sigma)$, write $\Delta(L)$ for the following sentence in the language $L_R$:
\[\bigwedge_{\phi=_L\psi} (\vert \phi/\psi\vert=1\wedge \vert\psi/\phi\vert=1) \; \wedge \bigwedge_{\phi>_L\psi} (\vert \phi/\psi\vert>1\wedge \vert \psi/\phi\vert=1).\]
Note that an $R$-module $M$ satisfies $\Delta(L)$ if and only if the ordering of the pp-formulae in $\Sigma$ given by $L$ is the same as the inclusion ordering of the sets they define in $M$.

Recall that, \ref{Tuganbaev}, when $R$ is a Pr\"ufer domain, for all $a,b\in R$ there exist $\alpha,r,s\in R$ such that $b\alpha=as$ and $a(\alpha-1)=br$ (for technical reasons, see the next Lemma, we swap here $a$ with $b$ and $r$ with $s$). Moreover, if $R$ is effectively given, then, given $a,b\in R$, we can effectively find such $\alpha, r, s\in R$.

\begin{lemma}\label{apTuganbaev}
Let $a,b\in R$, and let $\alpha, r, s \in R$ satisfy $a\alpha=br$ and $b(\alpha-1)=as$. For all $R$-modules $M$,
\begin{enumerate}[(i)]
\item if $\vert x\alpha=0 \, / \, x=0 \, (M) \vert = 1$ then $\exists y \ (ya=x\wedge yb=0)$ is equivalent to $x=0$ in $M$,
\item if $ \vert x(\alpha-1)=0 \, / \, x=0 \, (M) \vert=1$ then $\exists y \ (ya=x\wedge yb=0)$ is equivalent to $a|x\wedge xs=0$ in $M$,
\item if $\vert x=x \, / \, \alpha |x \, (M) \vert=1$ then $a|xb$ is equivalent to $r|x+xb=0$ in $M$ and
\item if $\vert x=x \, / \, (\alpha-1) |x \, (M) \vert=1$ then $a|xb$ is equivalent to $x=x$ in $M$.
\end{enumerate}
\end{lemma}
\begin{proof}
(i) Suppose that $M$ satisfies $\vert x\alpha=0 \, / \, x=0\vert=1$. Let $m,m'\in M$ be such that $m'a=m$ and $m'b=0$. Then $0=m'br=m'a\alpha=m\alpha$. So $m=0$.

(ii) Suppose that $M$ satisfies $\vert x(\alpha-1)=0 \, / \, x=0\vert=1$. Let $m,m'\in M$ be such that $m'a=m$ and $m'b=0$. Then $a|m$ and $ms=m'as=m'b(\alpha-1)=0$.

Let $m,m'\in M$ be such that $m=m'a$ and $ms=0$. Then $m'as=ms=0$. So $m'b(\alpha-1)=0$. Since $M$ satisfies $\vert x(\alpha-1)=0 \, / \, x=0 \vert=1$, $m'b=0$. So $m$ satisfies $\exists y \ (ya=x\wedge yb=0)$.

(iii) Suppose that $M$ satisfies $\vert x=x \, / \, \alpha |x\vert=1$. Let $m,m'\in M$ be such that $mb=m'a$. Since $M$ satisfies $\vert x=x \, / \, \alpha |x\vert=1$, there exists $m''\in M$ such that $m'=m''\alpha$. So $mb=m''\alpha a=m''br$. So $(m-m''r)b=0$ and hence $m$ satisfies $r|x+xb=0$.

Let $m,m',m''\in M$ be such that $m''b=0$ and $m=m'r+m''$. Then $mb=m'rb=m'a\alpha$. So $m$ satisfies $a|xb$.

(iv) Suppose $M$ satisfies $\vert x=x \, / \, (\alpha-1) |x\vert=1$. Let $m\in M$. Then there exists $m'\in M$ such that $m'(\alpha-1)=m$. So $m'as=m'(\alpha-1)b=mb$. Therefore $a|mb$.
\end{proof}

The next lemma will also be useful later.

\begin{lemma}\label{simppairs} Let $M$ be an $R$-module, $a,b,c,d\in R$. Then
\begin{enumerate}[(i)]
\item $\vert a|x \, / \, c|x \, (M) \vert =  \vert x = x \, / \, c | xa \, (M) \vert$,
\item $\vert a|x \, / \, xd=0 \, (M) \vert = \vert x = x \, / \, x a d = 0 \, (M) \vert$,
\item $\vert xb=0 \, / \, xd=0 \, (M) \vert = \vert \, \exists y \ (x=yd\wedge yb=0) \, / \, x=0 \, (M) \vert$,
\item $\vert x=x \, / \, xd = 0 \, (M) \vert = \vert d|x \, / \, x = 0 \, (M)\vert$.
\end{enumerate}
\end{lemma}

\begin{proof}
(i) and (ii) follow from considering the abelian group homomorphism from $M$ to $Ma / Mc$ (respectively to $Ma / \ann_M(d)$) which sends any $m \in M$ to the coset of $ma$ ($\ann_M (d)$ denotes here the annihilator of $d$ in $M$, that is, the pp-subgroup of the realizations in $M$ of $xd = 0$).

(iv) uses the scalar multiplication by $d$ in $M$.

For (iii) consider the abelian group homomorphism from $\ann_M (b)$ to the pp-subgroup of $M$ defined by $\exists y \ (x=yd \wedge yb=0)$ which sends any $m$ to $md$. This homomorphism is clearly surjective and $m \in \ann_M (b)$ is in its kernel if and only $md=0$.
\end{proof}

Now let $X,Y$ be non-empty finite subsets of $R$. Let $\Omega(X)$ (respectively $\Omega(Y)$) be the set of functions $P:X\rightarrow \{1,-1\}$ (respectively $Q:Y\rightarrow \{1,-1\}$). For each $(P,Q)\in \Omega(X)\times \Omega(Y)$, write $\Theta(P,Q)$ for the following sentence in the language $L_R$ (with $\alpha$ ranging over $X$ and $\beta$ over $Y$):
\begin{multline*}
\bigwedge_{P(\alpha)=1}\vert x\alpha=0 \, / \, x=0\vert=1 \wedge \bigwedge_{P(\alpha)=-1}\vert x(\alpha-1)=0 \, / \, x=0\vert=1 \wedge\\\bigwedge_{Q(\beta)=1}\vert x=x \, / \, \beta|x\vert=1\wedge\bigwedge_{Q(\beta)=-1}\vert x=x \, / \, (\beta-1)|x\vert=1.
\end{multline*}

Note that $\vert x\alpha=0 \, / \, x=0\vert=1$ is satisfied by an $R$-module $N$ if and only if $\alpha\notin \Ass \, N$ and $\vert x=x \, / \, \beta|x\vert=1$ is satisfied by an $R$-module $N$ if and only if $\beta\notin \Div \, N$.

Let us also point out that the only pairs of pp-formulae occurring in $\Theta (P, Q)$ are of the form required by Theorem \ref{reduction}.

\begin{lemma}\label{sigforPI}
Let $X,Y$ be non-empty finite subsets of $R$. If $N$ is an indecomposable pure injective $R$-module then there exists $(P,Q)\in \Omega(X)\times \Omega(Y)$ such that $N$ satisfies $\Theta(P,Q)$.
\end{lemma}
\begin{proof}
For $N$ an indecomposable pure injective $R$-module, $\Ass \, N$ and $\Div \, N$ are proper ideals. Thus for every $\alpha\in X$ (respectively $\beta\in Y$), either $\alpha\notin \Ass \, N$ (respectively $\beta\notin \Div \, N$) or $\alpha-1\notin \Ass \, N$ (respectively $\beta-1\notin \Div \, N$). Let $P:X\rightarrow \{1,-1\}$ (respectively $Q:Y\rightarrow \{1,-1\}$) be such that $P(\alpha)=1$ (respectively $Q(\beta)=1$) if $\alpha\notin \Ass \, N$ (respectively $\beta\notin \Div \, N$) and $P(\alpha)=-1$ (respectively $Q(\beta)=-1$) otherwise. Then $N$ satisfies $\Theta(P,Q)$.
\end{proof}

Now we are able to prove Theorem \ref{reduction}.


\begin{proof}
\noindent
\textbf{Step 1:} Let $\chi$ be the sentence labeled $(\star)$ at the beginning of this section and let $\Sigma$ be any finite non-empty set of pp-formulae.

 Note that there is an $R$-module satisfying $\chi$ if and only if there exists a non-empty subset $T$ of $\Gamma(\Sigma)$ and for each $L\in T$, an $R$-module $M_L$ satisfying $\Delta(L)$ such that $\bigoplus_{L\in T}M_L$ satisfies $\chi$.

The reverse direction is clear. Conversely, if there exists an $R$-module satisfying $\chi$, then there exists a finite direct sum of indecomposable pure injective (hence pp-uniserial by Fact \ref{ppuni}) modules satisfying $\chi$. Take $T$ to be the set of total orderings of $\Sigma$ determined by the inclusion of pp-subgroups in these direct summands.

Suppose that $T\subseteq\Gamma(\Sigma)$ is non-empty. There exist $R$-modules $M_L$ ($L \in T$) satisfying $\Delta(L)$ such that $\bigoplus_{L\in T}M_L$ satisfies $\chi$ if and only if the following conditions hold:
\begin{enumerate}
\item for each $1\leq i \leq t$, $\prod_{L\in T}\vert \phi_{1,i} (M_L) / \psi_{1,i} (M_L)\vert =H_i$,
\item for each $1\leq j \leq u$ and $L\in T$, $\vert \phi_{2,j} (M_L) / \psi_{2,j} (M_L)\vert =1$,
\item for each $1\leq k \leq s$, $\prod_{L\in T}\vert \phi_{3,k} (M_L) / \psi_{3,k}(M_L)\vert \geq E_k$.
\end{enumerate}
For each $1\leq i\leq t$, let $F^T_i$ be the set of functions $f:T\rightarrow \N$ such that $\prod_{L\in T}f(L)= H_i$. For each $1\leq k\leq s$, let $G^T_k$ be the set of functions $g:T\rightarrow \N$ such that $\prod_{L\in T}g(L)\geq E_k$ and for each $L\in T$, $g(L)\leq E_k$.

For each pair of tuples $f:=(f_1, \ldots, f_t)$ and $g:=(g_1, \ldots, g_s)$ with $f_i \in F^T_i$ and $g_k\in G^T_k$ and each $L\in T$, let $\chi^{L}_{(f,g)}$ be the sentence
\[\Delta(L)\wedge\bigwedge_{i=1}^t \vert \phi_{1,i} / \psi_{1,i} \vert = f_i(L) \,
\wedge \bigwedge_{j=1}^u \vert \phi_{2,j} / \psi_{2,j} \vert = 1 \; \wedge
\] \[ \wedge
\bigwedge_{k=1}^s \vert \phi_{3,k} / \psi_{3,k} \vert\geq g_k(L).\]

Now, there exists an $R$-module $M$ satisfying $\chi$ if and only if the following exist
\begin{enumerate}
\item $T\subseteq \Gamma(\Sigma)$ non-empty,
\item a pair of tuples $f:=(f_1,\ldots,f_l)$ and $g:=(g_1,\ldots,g_s)$ with $f_i\in F^T_i$ and $g_k\in G^T_k$,
\item for each $L\in T$, an $R$-module $M_L$ satisfying $\chi^{L}_{(f,g)}$.
\end{enumerate}

Since $R$ is a Pr\"ufer domain, we may assume that each $\phi_{S,i}$ (with $S=1, 2, 3$ and $i$ ranging
over the corresponding indices) is of the form
$\sum_{v=1}^{A_{S,i}} \, \exists y \ ( ya^S_{iv}=x\wedge y b^S_{iv}= 0)$
and each $\psi_{S,i}$ is of the form $\bigwedge_{g=1}^{B_{S,i}} (c^S_{iw} | xd^S_{iw})$,
where the involved scalars are elements of $R$.

Let $\Sigma$ be the set of formulae $\exists y \ (ya^S_{iv}=x\wedge y b^S_{iv}= 0)$ and $c^S_{iw} | xd^S_{iw}$  where $S \in \{1,2,3\}$, $1\leq i\leq t$ if $S=1$, $1\leq i\leq u$ if $S=2$, $1\leq i\leq s$ if $S=3$ and $1\leq v\leq A_{S, i}$, $1 \leq w \leq B_{S,i}$.

Let $T\subseteq \Gamma(\Sigma)$ and $L\in T$. Let $f:=(f_1,\ldots,f_l)$ and $g:=(g_1,\ldots,g_s)$ be a pair of tuples with $f_i\in F^T_i$ and $g_k\in G^T_k$. For each $\phi_{S,i}$ there exists $\sigma_{S,i} \in \Sigma$ such that $\Delta(L)\vdash \phi_{S,i} \leftrightarrow \sigma_{S,i}$ and for each $\psi_{S,i}$ there exists $\tau_{S,i} \in \Sigma$ such that $\Delta(L)\vdash \psi_i^s\leftrightarrow \tau_i^s$, moreover
each $\sigma_{S,i}$, $\tau_{S,i}$ can be effectively obtained from the corresponding
$\phi_{S,i}$, $\psi_{S,i}$. Thus $\chi^{L}_{(f,g)}$ is equivalent to

\[\Delta(L) \wedge \bigwedge_{i=1}^l \vert \sigma_{1,i} / \tau_{1,i} \vert=f_i(L) \wedge \bigwedge_{j=1}^u \vert \sigma_{2,j} / \tau_{2,j} \vert = 1 \wedge \bigwedge_{k=1}^s \vert \sigma_{3,k} / \tau_{3,k} \vert \geq g_k(L).\]

Thus, in order to show that the theory of $R$-modules is decidable, it is enough that there is an algorithm which given a sentence as in $(\star)$ with each $\phi_{S,i}$  of the form $\exists y \ (ya^S_{i}=x\wedge y b^S_{i}= 0)$ and each $\psi_{S,i}$ of the form $c^S_i|xd^S_i$, answers whether there exists an $R$-module satisfying it.

\medskip

\noindent
\textbf{Step 2: }Let $\chi$ be the sentence labeled $(\star)$, as reduced at the end of Step 1. Let $X,Y$ be non-empty finite subsets of $R$. Note that there is an $R$-module satisfying $\chi$ if and only if there exists a non-empty subset $T$ of $\Omega(X)\times \Omega(Y)$ and for each $(P,Q)\in T$, there exists an $R$-module $M_{(P,Q)}$ satisfying $\Theta(P,Q)$ such that $\bigoplus_{(P,Q)\in T}M_{(P,Q)}$ satisfies $\chi$.

This follows from Lemma \ref{sigforPI} since if there exists an $R$-module satisfying $\chi$ then there exists a finite direct sum of indecomposable pure injective $R$-modules satisfying $\chi$ and if two modules satisfy $\Theta(P,Q)$ then so does their direct sum.

Let $F_i^T$ and $G_k^T$ be as in Step 1, but adapted to the new setting where the (quasi-)orders $L$ of some subset of $\Gamma(\Sigma)$ are replaced by a subset of pairs $(P,Q)$ in $\Omega(X)\times \Omega(Y)$.
For each $f_i\in F_i^T$ and $g_k\in G_k^T$, let $\chi^{(P,Q)}_{(f,g)}$ be the sentence
\[\Theta(P,Q)\wedge\bigwedge_{i=1}^t \vert \phi_{1,i} / \psi_{1,i} \vert = f_i(P,Q) \,
\wedge \]
\[ \bigwedge_{j=1}^u \vert \phi_{2,j} / \psi_{2,j} \vert = 1 \; \wedge
\bigwedge_{k=1}^s \vert \phi_{3,k} / \psi_{3,k} \vert\geq g_k(P,Q).\]

Now, there exists an $R$-module $M$ satisfying $\chi$ if and only if the following exist:
\begin{enumerate}
\item $T\subseteq \Omega(X)\times \Omega(Y)$ non-empty,
\item a pair of tuples $f:=(f_1,\ldots,f_l)$ and $g:=(g_1,\ldots,g_s)$ with $f_i\in F^T_i$ and $g_k\in G^T_k$,
\item for each $(P,Q) \in T$, an $R$-module $M_L$ satisfying $\chi^{(P,Q)}_{(f,g)}$.
\end{enumerate}

Using Step 1, we may assume that each $\phi_{S,i}$ (with $S=1, 2, 3$ and $i$ ranging
over the corresponding indices) is of the form $\exists y \ (ya^S_{i}=x\wedge x b^S_{i}= 0)$ and each $\psi_{S,i}$ is of the form $c^S_i|xd^S_i$. For each $a^S_i, b^S_i$, let $\alpha^S_i,\delta^S_i,\gamma^S_i$ be such that $a^S_i\alpha^S_i=b^S_i\delta^S_i$ and $b(\alpha^S_i-1)=a_i^S\gamma_i^S$. For each $c^S_i, d^s_i$, let $\beta_i^S, \lambda_i^S, \mu_i^S$ be such that $c_i^S\beta_i^S=d_i^S\lambda_i^S $ and $d_i^S(\beta_i^S-1)=d_i^S\mu_i^S $. By Fact \ref{Tuganbaev}, such $\alpha^S_i,\delta^S_i,\gamma^S_i$ and $\beta_i^S, \lambda_i^S, \mu_i^S$ exist and if $R$ is effectively given then we can find them by searching.  Let $X$, $Y$ be the sets of the $\alpha_i^S$ and the $\beta_i^S$, respectively, where $S=1,2,3$ and $i$ ranges over the corresponding indices.

Let $T\subseteq \Omega(X)\times \Omega(Y)$ and $(P,Q)\in T$. Let $f:=(f_1,\ldots,f_l)$ and $g:=(g_1,\ldots,g_s)$ be a pair of tuples with $f_i\in F^T_i$ and $g_k\in G^T_k$. By Lemma \ref{apTuganbaev}, for each $\phi_{S,i}$, there exists a formula $\sigma_{S,i}$ of the form $a|x\wedge xs=0$ such that $\Theta(P,Q)\vdash \phi_{S,i}\leftrightarrow \sigma_{S,i}$ and for each $\psi_{S,i}$, there exists a formula $\tau_{S,i}$ of the form $r|x+xd=0$ such that $\Theta(P,Q)\vdash \psi_{S,i}\leftrightarrow \tau_{S,i}$ (and
there are algorithms producing these formulae). Thus $\chi^{(P,Q)}_{(f,g)}$ is equivalent to

\[\Theta(P,Q) \wedge \bigwedge_{i=1}^l \vert \sigma_{1,i} / \tau_{1,i} \vert=f_i(L) \wedge \]
\[  \bigwedge_{j=1}^u \vert \sigma_{2,j} / \tau_{2,j} \vert = 1 \wedge \bigwedge_{k=1}^s \vert \sigma_{3,k} / \tau_{3,k} \vert \geq g_k(L).\]

Thus, in order to show that the theory of $R$-modules is decidable, it is enough that there is an algorithm which given a sentence as in $(\star)$ with each $\phi_{S,i}$ of the form $a|x \, \wedge \, xs=0$ and each $\psi_{S,i}$ of the form $r|x \, + \, xd=0$, answers whether there exists an $R$-module satisfying it.

\medskip

\noindent
\textbf{Step 3: }Let $\chi$ be as in $(\star)$ with $\phi_{S,i}$ equal to $a_i^S|x \, \wedge \, xs_i^S=0$ and $\psi_{S,i}$ equal to $r_i^S|x \, + \, xd_i^S=0$ with $a_i^S,s_i^S,r_i^S,d_i^S\in R$.

Proceeding as in Step 1 with $\Sigma$ equal to the set of formulae $a_i^S|x$, $xs_i^S=0$, $r_i^S|x$ and $xd_i^S=0$ one can show that the theory of $R$-module is decidable if and only if there is an algorithm which, given a sentence
\[\bigwedge_{i=1}^t \vert\phi_{1,i} / \psi_{1,i} \vert= H_i \wedge \bigwedge_{j=1}^u \vert \phi_{2,j} / \psi_{2,j} \vert = 1 \wedge \bigwedge_{k=1}^s \vert \phi_{3,k} / \psi_{3,k} \vert \geq E_k \]
where for $i = 1, \ldots, t$, $j = 1, \ldots, u$ and $k = 1, \ldots, s$, $H_i$, $E_k$ are integers $\geq 2$
and the pp-formulae $\phi_{1,i}, \psi_{1,i}, \phi_{2,j},\psi_{2,j}, \phi_{3,k}, \psi_{3,k}$ are of the form $a|x$ or $xb=0$ with $a, b \in R$, answers whether there exists an $R$-module satisfying this sentence.

\medskip

\noindent
\textbf{Step 4: }
Let $\chi$ be of the form we reduced to at the end of Step 3. By Lemma \ref{simppairs}, we can replace in $\chi$
\begin{enumerate}[(i)]
\item every instance of the form $| \, a|x \, / \, c|x \, |$ by $| \, x=x \, / \, c|xa \, |$, 
\item every instance of the form $|\, a|x \, / \, xd=0 \, |$ by $|\, x=x \, / \, xad=0 \, |$
\item and every instance of the form $|\, xb=0 \, / \, xd=0|$ by $|\, \exists y \ (x=yd\wedge yb=0) \, / \, x=0
\,|$.
\end{enumerate}
Repeating Step 2 and recalling that only pairs of the form $x\alpha=0 \, / \, x=0$ and $x=x \, / \, \beta|x$ occur in the sentences $\Theta(P,Q)$, we are led to consider a conjunction of invariants sentences involving only pairs of the form $x=x \, / \, \rho | x +  x \sigma =0 $, $ \rho |x \wedge x \sigma =0 \, / \, x=0$, $x=x \, / \, xd=0$ and $xb=0 \, / \, c|x$. So we can assume that $\chi$ is a conjunction of invariants sentences involving only pairs of this form.

\medskip

\noindent
\textbf{Step 5: }
Suppose that a pair of the form $x=x \, / \, \rho |x +  x \sigma=0$ or $\rho |x  \wedge xs=0 \, / \, x=0$ occurs in $\chi$ for some $\rho, \sigma \in R$.
Put $\Sigma:=\{\rho |x, x \sigma =0\}$ and take $L\in \Gamma(\Sigma)$. Then the only pairs that occur in $\Delta(L)$ are $x \sigma =0 / \rho | x$, which is already of the required final form in the statement of the theorem, and $\rho | x \, / \, xs=0$, which by Lemma \ref{simppairs} can be replaced by $x=x \, / \, x \rho \sigma =0$. Hence all pairs occurring in $\Delta(L)$ are of the required form.

Repeating Step 1 of the proof with $\Sigma=\{ \rho |x , \, x \sigma =0 \}$ produces sentences $\chi_{(f,g)}^L$ where we can replace each instance of $x=x \, / \, \rho | x + x \sigma =0 $ by $x=x \, / \, \rho|x$ or $x=x \, / \, x \sigma =0$ as appropriate and each instance of $\rho |x  \wedge x \sigma =0 \, / \, x=0$ by $\rho |x \, / \, x=0$ or $ x \sigma =0 \, / \, x=0$ as appropriate. By Lemma \ref{simppairs}, (iv), we may replace all instances of the pair $\rho | x \, / \, x=0$ by $x=x \, / \, x \rho =0$.  Repeating this process for each $\rho, \sigma \in R$ such that the pair $x=x \, / \, \rho | x + x \sigma =0$ or the pair $\rho |x \wedge x \sigma=0 \, / \, x=0$ occurs in $\chi$ allows us to reduce to considering sentences of the form required by the statement of the theorem.
\end{proof}

\section{Preparatory lemmas}\label{prep}

We assume throughout this section that $R$ is a Pr\"ufer domain such that all the localizations of $R$ at maximal ideals have dense value group.

The focus of this section will be the $R$-modules $$N_{\gamma}(\mfrak{m}):=R_{\mfrak{m}}/\gamma \mfrak{m}R_{\mfrak{m}}, \, \, N'_{\beta, \eta}(\mfrak{m}):=\mfrak{m}R_{\mfrak{m}}/ \beta \eta R_{\mfrak{m}}$$ where $\mfrak{m}$ is a maximal ideal of $R$, $\gamma\in R\backslash\{0\}$ and $\beta, \eta\in\mfrak{m}\backslash\{0\}$.

It was shown in \cite[Proposition 7.8]{PPT} that, over a valuation domain $V$ with dense value group and {\bf finite residue field}, the only indecomposable pure injective modules $N$ such that there exists a pp-pair $\phi/\psi$ with $\vert\phi/\psi(N)\vert$ finite but not equal to $1$ are those corresponding to the types $(\beta V, \eta V)$ and $(\mfrak{p},\gamma\mfrak{p})$ where $\mfrak{p}$ is the maximal ideal of $V$, $\gamma\in V\backslash\{0\}$ and $\beta, \eta \in\mfrak{p}\backslash\{0\}$. These types are realized in the uniserial $V$-modules $\mfrak{p}/ \beta \eta V$ and $V/\gamma\mfrak{p}$. Thus all such indecomposable pure injective modules are of the form $\text{PE}(\mfrak{p}/ \beta \eta V)$ or $\text{PE}(V/\gamma\mfrak{p})$
(recall that PE means pure injective hull). If the residue field of $V$ is not finite then no such indecomposable pure injective modules exist.

If $N$ is an indecomposable pure injective module over a Pr\"{u}fer domain $R$ then there exists some maximal ideal $\mfrak{m}$ such that $N$ is the restriction of an indecomposable pure injective $R_{\mfrak{m}}$-module. Now, if there exists a pp-pair $\phi/\psi$ such that $\vert\phi/\psi(N)\vert$ is
finite but not equal to $1$ then $N$ is either of the form $\text{PE}(\mfrak{m}R_\mfrak{m}/\beta \eta R_\mfrak{m})$ or of the form $\text{PE}(R_\mfrak{m}/\gamma\mfrak{m}R_\mfrak{m})$ where $R/\mfrak{m}$ is finite, $\beta,\eta\in\mfrak{m}R_{\mfrak{m}}\backslash\{0\}$ and $\gamma\in R_{\mfrak{m}}\backslash\{0\}$. Since all elements of $R_{\mfrak{m}}$ are unit multiples of elements in $R$, we may assume that $\beta,\eta\in\mfrak{m}\backslash\{0\}$ and $\gamma\in R\backslash\{0\}$.

Finally, for $R$ any commutative ring and $\mfrak{m}$ a maximal ideal of $R$, if $M$ is a module over $R_{\mfrak{m}}$, then
taking the pure injective hull of $M$ over $R_\mfrak{m}$ and then restricting to $R$ is the same as taking the pure injective hull of $M$ as
an $R$-module.

We will need the following result from \cite{PPT}. Recall that a pp-pair $\phi / \psi$ is {\sl minimal}
(in the theory of a given module $N$ over any ring) if $\phi(N)$ properly includes its intersection with $\psi(N)$ and there is no intermediate pp-subgroup $\theta(N)$ such that $\phi(N) \supsetneq \theta (N) \supsetneq \phi(N) \cap \psi(N)$.

\begin{lemma}\label{finiteimpminpair}
(\text{\cite[Lemma 7.5 and Corollary 7.6]{PPT}}).
Let $V$ be a commutative valuation domain and $\phi/\psi$ be a pp-$1$-pair over $V$. If $N$ is an indecomposable pure injective $V$-module and $\vert \, \phi/\psi(N) \, \vert$ is finite and $> 1$ then $\phi/\psi$ is an $N$-minimal pair. Moreover, if $\mfrak{p}$ is
the maximal ideal of $V$, then $\phi(N)/\psi(N)$ is a 1-dimensional
vector space over the residue field $V/\mfrak{p}$, that consequently is finite.
\end{lemma}

When $R$ is a Pr\"ufer domain, and so every localization at a maximal ideal is
a commutative valuation domain, we obtain the following consequence. Let $N$ be
an indecomposable pure injective module over $R$, and let $\mfrak{m}$ be a maximal
ideal of $R$ such that $N$ is a module over $R_\mfrak{m}$. Then every pp-1-pair
$\phi / \psi$ over $R$ with $| \, \phi / \psi (N) \, |$ finite and greater than 1 is $N$-minimal
and $\phi / \psi (N)$ is a 1-dimensional vector space over the residue field $R / \mfrak{m}$,
which must therefore be finite.

The minimal pairs of modules, over a valuation domain $V$ with maximal ideal $\mfrak{p}$ and dense value group, of the form $V/\gamma\mfrak{p}$ and $\mfrak{p}/\beta \eta V$ were described in \cite[Section 7]{PPT}  at least for valuation domains with finite residue fields. However, the results in Section \ref{order} focus our interest on pp-pairs of the form $xb=0 \, / \, c|x$ and $x=x \, / \, xd=0$. We will now prove the results about minimal pairs which we need without the assumption that $V$ has finite residue field.

The following fact can be derived from \cite[Theorem 4.3]{Gre13}.
\begin{fact}\label{opensetvdom}
Let $V$ be a valuation domain and $(I,J)$ be a pair of proper ideals in $V$. Then
$\text{PE}(I,J)\in \left(xb =0 \wedge a|x \, / \, xd=0 + c|x \right)$ if and only if
$a\neq 0$, $d\neq 0$, $c \in a J^\#$, $b\in dI^\#$, $bc \in IJ$ and $ad \notin \ann_V \text{PE}(I,J)$.
\end{fact}

 For $I$ an ideal of $V$, $I^\#$ denotes $\bigcup_{r\in V\backslash I}(I:r)$. Note that $\Ass \; \text{PE}(I,J)=I^\#$ and $\Div \; \text{PE}(I,J)=J^\#$. For $V$ a valuation domain with maximal ideal $\mfrak{p}$, $\gamma\in V\backslash\{0\}$ and $\beta \in\mfrak{p}\backslash\{0\}$, $(\gamma \mfrak{p})^\#=\mfrak{p}$ and $(\beta V)^\#=\mfrak{p}$.

For $\gamma\in R\backslash\{0\}$, the pure injective hull of $R_\mfrak{m}/\gamma\mfrak{m}R_{\mfrak{m}}$ corresponds to the pair $(\gamma\mfrak{m}R_{\mfrak{m}},\mfrak{m}R_{\mfrak{m}})$ of ideals of $R_{\mfrak{m}}$. The annihilator, as an $R_{\mfrak{m}}$-module, of $R_\mfrak{m}/\gamma\mfrak{m}R_{\mfrak{m}}$, and hence $\text{PE}(\gamma\mfrak{m}R_{\mfrak{m}},\mfrak{m}R_{\mfrak{m}})$, is $\gamma\mfrak{m}R_{\mfrak{m}}$.

For $0 \neq \delta \in \mfrak{m}$, the pure injective hull of $\mfrak{m}R_{\mfrak{m}}/\delta R_{\mfrak{m}}$ corresponds to a pair $(\beta R_{\mfrak{m}},\eta R_{\mfrak{m}})$ of ideals of $R_{\mfrak{m}}$ where $\beta,\eta\in \mfrak{m}$ and $\beta \eta R_{\mfrak{m}}=\delta R_{\mfrak{m}}$. The annihilator, as an $R_{\mfrak{m}}$-module, of $\mfrak{m}R_{\mfrak{m}}/\delta R_{\mfrak{m}}$, and hence of $\text{PE}(\beta R_{\mfrak{m}},\eta R_{\mfrak{m}})$, is $\delta R_{\mfrak{m}}$.

For $0\neq \delta\in \mfrak{m}$, the pure injective hull of $R_{\mfrak{m}}/\delta R_{\mfrak{m}}$ corresponds to the pair $(\delta R_{\mfrak{m}},\mfrak{m}R_{\mfrak{m}})$ of $R_{\mfrak{m}}$ ideals. The annihilator, as an $R_{\mfrak{m}}$-module, of $R_{\mfrak{m}}/\delta R_{\mfrak{m}}$, and hence of $\text{PE}(\delta R_{\mfrak{m}},\mfrak{m}R_{\mfrak{m}})$, is $\delta R_{\mfrak{m}}$.

Note that this means that if $\delta\in\mfrak{m}\backslash\{0\}$ then $\text{PE}(\delta R_{\mfrak{m}},\mfrak{m})$ is in the Ziegler closure of both $\text{PE}(\delta\mfrak{m}R_{\mfrak{m}},\mfrak{m}R_{\mfrak{m}})$ and $\text{PE}(\beta R_{\mfrak{m}},\eta R_{\mfrak{m}})$ where $\delta R_{\mfrak{m}}=\beta \eta R_{\mfrak{m}}$ and $\beta, \eta\in\mfrak{m}\backslash\{0\}$. Since $\Zg_{R_{\mfrak{m}}}$ embeds homeomorphically into $\Zg_{R}$ as a closed subset, it doesn't matter whether we take closures in $\Zg_R$ or $\Zg_{R_\mfrak{m}}$.

\begin{lemma}\label{condpp}
Let $b,c,d\in R$, $\mfrak{m}$ be a maximal ideal of $R$, $\gamma\in R\backslash\{0\}$ and $\delta\in\mfrak{m}\backslash\{0\}$.
\begin{enumerate}
\item $R_\mfrak{m}/\gamma\mfrak{m}R_{\mfrak{m}}$ opens $x=x\, / \, xd=0$ if and only if $\gamma\in dR_{\mfrak{m}}$.
\item $R_{\mfrak{m}}/\gamma\mfrak{m}R_{\mfrak{m}}$ opens $xb=0 \, / \, c|x$ if and only if $c \in\mfrak{m}R_{\mfrak{m}}$, $b\in\mfrak{m}R_{\mfrak{m}}$ and $bc\in\gamma\mfrak{m}R_{\mfrak{m}}$.
\item $\mfrak{m}R_{\mfrak{m}}/\delta R_{\mfrak{m}}$ opens $x=x \, / \, xd=0$ if and only if $\delta\in d\mfrak{m}R_{\mfrak{m}}$.
\item $\mfrak{m}R_{\mfrak{m}}/\delta R_{\mfrak{m}}$ opens $xb=0 \, / \, c|x$ if and only if $c \in\mfrak{m}R_{\mfrak{m}}$, $b\in\mfrak{m}R_{\mfrak{m}}$ and  $bc \in \delta R_{\mfrak{m}}$.
\end{enumerate}
\end{lemma}
\begin{proof}
Each claim can be deduced directly from Fact \ref{opensetvdom}.
\end{proof}

Moreover, Fact \ref{opensetvdom}
implies that for $\gamma\in\mfrak{m}\backslash\{0\}$, $R_{\mfrak{m}}/\gamma R_{\mfrak{m}}$ opens $x=x\, / \, xd=0$ if and only if $d\notin \gamma R_{\mfrak{m}}$, and $R_{\mfrak{m}}/\gamma R_{\mfrak{m}}$ opens $xb=0 \, / \, c|x$ if and only if $b \in \mfrak{m} R_{\mfrak{m}}$, $c \in\mfrak{m}R_{\mfrak{m}}$ and $bc \in \gamma \mfrak{m}R_{\mfrak{m}}$.

\begin{lemma}\label{condpp1}
Let $b,c,d\in R$, $\mfrak{m}$ be a maximal ideal of $R$, $\gamma\in R\backslash\{0\}$ and $\delta\in\mfrak{m}\backslash\{0\}$.
\begin{enumerate}
\item $x=x \, / \, xd=0$ is a minimal pair for $R_{\mfrak{m}}/\gamma\mfrak{m}R_{\mfrak{m}}$ if and only if $\gamma R_{\mfrak{m}}=dR_{\mfrak{m}}$.
\item $xb=0 \, / \,  c|x$ is a minimal pair for $R_{\mfrak{m}}/\gamma\mfrak{m}R_{\mfrak{m}}$ if and only if $\gamma\notin \mfrak{m}$, $b \in\mfrak{m}$ and $c \in\mfrak{m}$.
\item $x=x \, / \, xd=0$ is never a minimal pair for $\mfrak{m}R_{\mfrak{m}}/\delta R_{\mfrak{m}}$.
\item $xb=0 \, / \, c|x$ is a minimal pair for $\mfrak{m}R_{\mfrak{m}}/\delta R_{\mfrak{m}}$ if and only if $c \in\mfrak{m}R_{\mfrak{m}}$, $b\in\mfrak{m}R_{\mfrak{m}}$ and $bc R_{\mfrak{m}}=\delta R_{\mfrak{m}}$.
\end{enumerate}
\end{lemma}
\begin{proof}
Recall, \cite[Corollary 8.12]{Zi}, that if $N$ is an indecomposable pure injective $R$-module and $\phi/\psi$ is an $N$-minimal pair then $(\phi/\psi)$ isolates $N$ in its Ziegler closure.

\smallskip

\noindent
(1) If $\gamma R_{\mfrak{m}}=dR_{\mfrak{m}}$ then, for $r \in R_{\mfrak{m}}$,
$r+\gamma\mfrak{m}R_{\mfrak{m}}$ satisfies $xd=0$ if and only if $r\in\mfrak{m}R_{\mfrak{m}}$. Thus, as an $R$-module, $x=x/xd=0$ evaluated at $R_{\mfrak{m}}/\gamma\mfrak{m}R_{\mfrak{m}}$ is isomorphic to the simple $R$-module $R/\mfrak{m}$. So $x=x/xd=0$ is a minimal pair for $R_{\mfrak{m}}/\gamma\mfrak{m}R_{\mfrak{m}}$.

For the converse, suppose that $R_{\mfrak{m}}/\gamma\mfrak{m}R_{\mfrak{m}}$ opens $x=x/xd=0$, so $\gamma\in dR_{\mfrak{m}}$. If $d\notin \gamma R_{\mfrak{m}}$ then $\text{PE}(R_{\mfrak{m}}/\gamma R_{\mfrak{m}})$ opens $x=x/xd=0$. Since $\text{PE}(R_{\mfrak{m}}/\gamma R_{\mfrak{m}})$ is in the Ziegler closure of $\text{PE}(R_{\mfrak{m}}/\gamma\mfrak{m}R_{\mfrak{m}})$, this implies that $x=x/xd=0$ is not a $\text{PE}(R_{\mfrak{m}}/\gamma\mfrak{m}R_{\mfrak{m}})$-minimal pair and hence also not a $R_{\mfrak{m}}/\gamma\mfrak{m}R_{\mfrak{m}}$-minimal pair.

\smallskip

\noindent
(2) Suppose $\gamma\notin \mfrak{m}$, $b, c \in\mfrak{m}$. Since $\gamma\notin \mfrak{m}$, $R_{\mfrak{m}}/\gamma\mfrak{m}R_{\mfrak{m}}=R/\mfrak{m}$ is a simple $R$-module. Therefore $xb=0\, / \, c|x$ is a $R_{\mfrak{m}}/\gamma\mfrak{m}R_{\mfrak{m}}=R/\mfrak{m}$-minimal pair if and only if $R_{\mfrak{m}}/\gamma\mfrak{m}R_{\mfrak{m}}$ opens $xb=0 \, / \, c|x$. That $xb=0 \, / \, c|x$ is a $R_{\mfrak{m}}/\gamma\mfrak{m}R_{\mfrak{m}}=R/\mfrak{m}$-minimal pair now follows from (2) in
Lemma \ref{condpp}.

Suppose that $xb=0 \, / \, c|x$ is an $R_\mfrak{m}/\gamma \mfrak{m} R_{\mfrak{m}}$-minimal pair. Again from (2) in Lemma \ref{condpp}, $c \in\mfrak{m}R_{\mfrak{m}}$, $b\in\mfrak{m}R_{\mfrak{m}}$ and $bc \in\gamma \mfrak{m} R_{\mfrak{m}}$. Now, if $\gamma \in \mfrak{m}$ then $\text{PE}(R_{\mfrak{m}}/\gamma R_{\mfrak{m}})$ opens $xb=0 \, / \, c|x$. Hence $xb=0 \, / \, c|x$ is not a $R_\mfrak{m}/\gamma \mfrak{m} R_{\mfrak{m}}$-minimal pair. So $\gamma \notin \mfrak{m}$.

\smallskip

\noindent
(3) The module $\mfrak{m}R_{\mfrak{m}}/\delta R_{\mfrak{m}}$ opens $x=x \, / \, xd=0$ if and only if $d\notin \delta R_{\mfrak{m}}$, and $R_{\mfrak{m}}/\delta R_{\mfrak{m}}$ opens $x=x/xd=0$ if and only if $d\notin \delta R_{\mfrak{m}}$. Since $\text{PE}(R_{\mfrak{m}}/\delta R_{\mfrak{m}})$ is in the closure of $\text{PE}(\mfrak{m}R_{\mfrak{m}}/\delta R_{\mfrak{m}})$, $x=x \, / \, xd=0$ is never a $\mfrak{m}R_{\mfrak{m}}/\delta R_{\mfrak{m}}$-minimal pair.

\smallskip

\noindent
(4) Suppose that $c \in\mfrak{m}R_{\mfrak{m}}$, $b\in\mfrak{m}R_{\mfrak{m}}$ and $bc R_{\mfrak{m}}=\delta R_{\mfrak{m}}$. Then the solution set of $xb=0$ in $\mfrak{m}R_{\mfrak{m}}/\delta R_{\mfrak{m}}$ is $cR_{\mfrak{m}}/\delta R_{\mfrak{m}}$ and the solution set of $c|x$ in $\mfrak{m}R_{\mfrak{m}}/\delta R_{\mfrak{m}}$ is $c \mfrak{m}R_{\mfrak{m}}/\delta R_{\mfrak{m}}$. So $xb=0\, / \, c|x$ evaluated at $\mfrak{m}R_{\mfrak{m}}/\delta R_{\mfrak{m}}$ is the simple $R$-module $R/\mfrak{m}$ and hence $xb=0\, / \, c|x$ is a $\mfrak{m}R_{\mfrak{m}}/\delta R_{\mfrak{m}}$-minimal pair.

Suppose $xb=0\, / \, c|x$ is a $\mfrak{m}R_{\mfrak{m}}/\delta R_{\mfrak{m}}$-minimal pair. By (4) in
Lemma \ref{condpp}, $c \in \mfrak{m}R_{\mfrak{m}}$, $b\in \mfrak{m}R_{\mfrak{m}}$ and $bc \in\delta R_{\mfrak{m}}$. Suppose, for a contradiction, that $bc \in \delta \mfrak{m} R_{\mfrak{m}}$. Then $R_{\mfrak{m}}/\delta R_{\mfrak{m}}$ opens $xb=0 \, / \, c|x$. So, we can argue as in (1) and (2) that $xb=0 \, / \, c|x$ is not a $\mfrak{m}R_{\mfrak{m}}/\delta R_{\mfrak{m}}$-minimal pair.
\end{proof}

Now, still for $R$ a Pr\"ufer domain, let us come back to the indecomposable pure injective $R$-modules of the form $\text{PE}(\mfrak{m}R_\mfrak{m}/\beta \eta R_\mfrak{m})$ and $\text{PE}(R_\mfrak{m}/\gamma\mfrak{m}R_\mfrak{m})$ where $R/\mfrak{m}$ is finite, $\beta,\eta\in\mfrak{m}\backslash\{0\}$ and $\gamma\in R\backslash\{0\}$ (those admitting a pp-pair $\phi/\psi$ with $\vert\phi/\psi(N)\vert$ finite but not equal to $1$).

The value of $\vert\phi/\psi(N)\vert$ for a pp-pair $\phi/\psi$ when $N$ is one of the above $R_{\mfrak{m}}$-uniserial modules will be determined by conditions of the form $a\in bR_{\mfrak{m}}$ and $a\in b\mfrak{m}R_{\mfrak{m}}$ with $a, b \in R$. The following lemma, together with Fact \ref{Tuganbaev}, allows us to convert such conditions into conditions of the form $c\in \mfrak{m}$ (see also the previous
Remark \ref{rem}).

\begin{lemma}\label{relmaxideal}
Let $a,b\in R\backslash\{0\}$ and let $\alpha,r,s\in R$ be such that $b\alpha =as$ and $a(\alpha-1)=br$. Then
\begin{enumerate}
\item $b\in aR_{\mfrak{m}}$ if and only if $\alpha\notin\mfrak{m}$ or $r\notin\mfrak{m}$;
\item $a\in b\mfrak{m}R_{\mfrak{m}}$ if and only if $\alpha\in\mfrak{m}$ and $r\in\mfrak{m}$.
\end{enumerate}
\end{lemma}
\begin{proof}
(1) is \cite[Lemma 5.5]{GLPT}. (2) follows from (1) since $a\in b\mfrak{m}R_\mfrak{m}$ if and only if $b\notin aR_\mfrak{m}$.
\end{proof}

Note that, see \cite{GLPT} just after Lemma 5.5, over a B\'ezout domain things become even simpler.

This leads us to consider what we call a {\sl condition on a maximal ideal} (of $R$), that is, a condition of the form $r \in M$ where $r \in R$ and $M$ is a variable for a maximal ideal (of $R$). Let $\mathbb{B}$ denote the set of Boolean combinations of these conditions. We will say that a maximal ideal $\mfrak{m}$ of $R$ satisfies such a Boolean combination $\Delta$ if when we replace all instances of $M$ by $\mfrak{m}$,  $\Delta$ is true in $R$.

Any $\Delta\in\mathbb{B}$  is equivalent to a disjunction of conditions of the form
\[\bigwedge_{i=1}^l a_i\in M\wedge b\notin M\] for some $l \in\N$.

To see this first put $\Delta$ into disjunctive normal form and then note that a condition of the form $\bigwedge_{i=1}^l b_i\notin M$ is equivalent to $\prod_{i=1}^l b_i\notin M$.

Note that, when $R$ is B\'ezout, also a conjunction of the form $\bigwedge_{i=1}^l a_i\in M$ is equivalent to a single condition $a \in M$ where $a$ is the greatest common divisor of $a_1, \ldots, a_l$. So, when $R$ is B\'ezout, each $\Delta\in\mathbb{B}$ is equivalent to a disjunction of conditions of the form $a\in M \wedge b\notin M$.

Now, for every $\Delta$ in $\mathbb{B}$, let $\text{PP}_0 (R, \Delta)$ denote the set of all $(p,n)\in \mathbb{P}\times \N$ such that there exist $s, k_1,\ldots, k_s\in \N$ and maximal ideals $\mfrak{m}_1,\ldots,\mfrak{m}_s$ of $R$ such
that $n\in\text{Span}_{\N_0}\{k_1,\ldots,k_s\}$ and for all $i=1, \ldots, s$,
\begin{enumerate}
\item $|R/\mfrak{m}_i|=p^{k_i}$,
\item $\mfrak{m}_i$ satisfies $\Delta$.
\end{enumerate}

Let $\text{PP}_0 (R)$  be the set of all $(p,n,\Delta)\in \mathbb{P}\times\N\times\mathbb{B}$ such that $(p, n) \in \text{PP}_0 (R, \Delta)$.

\begin{lemma}\label{PPrecimpPP1rec}
Suppose that $PP^\star (R)$ is recursive. Then $PP_0 (R)$ is recursive.
\end{lemma}
\begin{proof}
Let $\Delta$ have the form $\bigvee_{i=1}^m (\bigwedge_{h=1}^la_{ih}\in M\wedge b_i\notin M)$ with $a_{ih}, b_i\in R$. This can be assumed without loss of generality, adding if necessary $0$ for $a_{ih}$ and $1$ for $b_i$.

We will now show that $(p,n,\Delta)\in \text{PP}_0 (R)$ if and only if there exists $(\delta_1,\ldots,\delta_m)\in (\N_0)^m$ such that $\sum_{i=1}^m\delta_i=n$ and for all $1\leq i\leq l$, either $(p,\delta_i,a_{i1},\ldots,a_{il},b_i)\in \text{PP}^\star (R)$ or $\delta_i=0$. Since the set of $(\delta_1,\ldots,\delta_m)\in (\N_0)^m$ such that $\sum_{i=1}^m\delta_i=n$ is finite and computable given $n$, this will imply that if $\text{PP}^\star (R)$ is recursive then so is $\text{PP}_0 (R)$.

Suppose that $n=\sum_{i=1}^m\delta_i$, each $\delta_i\in\N_0$ and for all $1\leq i\leq m$, either $(p,\delta_i, a_{i1},\ldots, a_{il}, b_i)\in\text{PP}^\star (R)$ or $\delta_i=0$. So, for each $1\leq i\leq m$ with $\delta_i\neq 0$, there exist $k_{i1},\ldots k_{is_i}\in\N$ such that $\delta_i\in \text{Span}_{\N_0}\{k_{i1},\ldots k_{is_i}\}$ and maximal ideals $\mfrak{m}_{i1},\ldots,\mfrak{m}_{is_i}$ such that $|R/\mfrak{m}_{ij}|=p^{k_{ij}}$, $a_{ih}\in \mfrak{m}_{ij}$ for $1\leq h\leq l$ and $b_i\notin\mfrak{m}_{ij}$. Thus, for $1\leq i\leq m$ with
$\delta_i \neq 0$ and $1\leq j\leq s_i$, $\mfrak{m}_{ij}$ satisfies $\Delta$, $|R/\mfrak{m}_{ij}|=p^{k_{ij}}$ and $n\in\text{Span}_{\N_0}\{k_{ij} \st 1\leq i\leq m \text{ and } 1\leq j\leq s_i\}$.

Now suppose that $(p,n)\in\mathbb{P}\times \N$ and that there exist $k_1,\ldots, k_s\in \N$ such that $n\in\text{Span}_{\N_0}\{k_1,\ldots,k_s\}$ and maximal ideals $\mfrak{m}_1,\ldots,\mfrak{m}_s$ such that $|R/\mfrak{m}_j|=p^{k_j}$ and $\mfrak{m}_j$ satisfies $\Delta$ for $1\leq j\leq s$. Let $\lambda_1,\ldots \lambda_s \in \N_0$ be such that $n=\sum_{j=1}^s \lambda_jk_j$. We may partition $\{1,\ldots,s\}$ into sets $A_1,\ldots A_m$ such that, for all $1 \leq j \leq s$ and $1 \leq i \leq m$, $j\in A_i$ implies that $\mfrak{m}_j$ satisfies $a_{ih}\in\mfrak{m}_j$ for $1\leq h\leq l$ and $b_i\notin\mfrak{m}_j$. Let $\delta_i:=\sum_{j\in A_i} \lambda_jk_j$. If $\delta_i \neq 0$ then $(p,\delta_i, a_{i1},\ldots a_{il} ,b_i)\in \text{PP}^\star (R)$ and $\sum_{i=1}^m\delta_i=n$ as required.
\end{proof}
The next definition describes the families of modules we are going to deal with. Indeed the summands of these families were already treated at least implicitly in this section.

\begin{definition}
\

\begin{enumerate}
\item[($S_\gamma$)] For $\gamma\in R\backslash\{0\}$,  let $S_\gamma$ be the set of $R$-modules of the form $\oplus_{i=1}^m N_\gamma (\mfrak{m_i})$ where $m \in \N$ and $\mfrak{m}_1,\ldots,\mfrak{m}_m$ are maximal ideals of $R$.
\item[($S'_{\beta, \eta}$)] For $\beta, \eta\in R\backslash\{0\}$,
let $S'_{\beta, \eta}$ be the set of $R$-modules of the form $\oplus_{i=1}^m N'_{\beta, \eta} (\mfrak{m_i})$
where $m \in \N$ and $\mfrak{m}_1,\ldots,\mfrak{m}_m$ are maximal ideals of $R$ containing both $\beta$
and $\eta$.
\item[($T_{\beta,\eta}$)] For $\beta, \eta\in R\backslash\{0\}$, let $T_{\beta,\eta}$ be the set of $R$-modules of the form $\oplus_{i=1}^m R/\mfrak{m}_i$ where $m \in \N$ and $\mfrak{m}_1,\ldots,\mfrak{m}_m$ are maximal ideals of $R$ containing both $\beta$
and $\eta$.
\end{enumerate}

\end{definition}

The preparation of the proof of the main theorem culminates in the next proposition.
\begin{proposition}\label{prepp}
Let $R$ be a Pr\"ufer domain such that each localization of $R$ at a maximal ideals has dense value group.
Suppose that $R$ is effectively given and $\text{PP}^\star (R)$ is recursive.

\begin{enumerate}[(a)]
\item Fix $\gamma \in R\backslash\{0\}$. Then there is an algorithm which, given
\begin{itemize}
\item $p\in\mathbb{P}$,
\item pp-pairs $\phi_i /  \psi_i$ for $1 \leq i \leq t+s$ and $\phi_{2,j}/ \psi_{2,j}$ for $1\leq j \leq u$ of the form $xb=0 \, / \, c|x$ or $x=x \, / \, xd=0$, with $b, c, d \in R$,
\item positive integers $w$, $n_i$ for $1\leq i\leq t$ and $l_i$ for $t+1\leq i \leq t+s$,
\end{itemize}
answers whether there exists an $R$-module $N\in S_\gamma$ satisfying the sentences $\vert \, x=x \, / \, x\gamma=0 \, \vert =p^w$ and
\[\bigwedge_{i=1}^{t}\vert\phi_i/\psi_i\vert=p^{n_i}
\wedge\bigwedge_{j=1}^u \vert\phi_{2,j} / \psi_{2,j} \vert=1
\wedge\bigwedge_{i=t+1}^{t+s} \vert\phi_i / \psi_i \vert \geq p^{l_i}. \]
\item Fix $\beta, \eta \in R\backslash\{0\}$. Then there is an algorithm which, given
\begin{itemize}
\item $p\in\mathbb{P}$,
\item pp-pairs $\phi_i/ \psi_i$ for $1 \leq i \leq t+s$ and $\phi_{2,j}/ \psi_{2,j}$ for $1\leq j \leq u$ of the form $xb=0 \, / \, c|x$ or $x=x \, / \, xd=0$, with $b, c, d \in R$,
\item positive integers $w$, $n_i$ for $1\leq i\leq t$ and $l_i$ for $t+1\leq i \leq t+s$,
\end{itemize}
answers whether there exists an $R$-module $M \in T_{\beta, \eta}$ satisfying the sentences $\vert \, x \eta = 0 \, / \, \beta |x \,  \vert =p^w$ and
\[\bigwedge_{i=1}^{t}\vert\phi_i/\psi_i\vert=p^{n_i}
\wedge\bigwedge_{j=1}^u \vert\phi_{2,j} / \psi_{2,j} \vert=1
\wedge\bigwedge_{i=t+1}^{t+s} \vert\phi_i / \psi_i \vert \geq p^{l_i}. \]

\item Fix $\beta, \eta \in R\backslash\{0\}$. Then there is an algorithm which, given
\begin{itemize}
\item $p\in\mathbb{P}$,
\item pp-pairs $\phi_i/ \psi_i$ for $1 \leq i \leq t+s$ and $\phi_{2,j}/ \psi_{2,j}$ for $1\leq j \leq u$ of the form $xb=0 \, / \, c|x$ or $x=x \, / \, xd=0$, with $b, c, d \in R$,
\item positive integers $w$, $n_i$ for $1\leq i\leq t$ and $l_i$ for $t+1\leq i \leq t+s$,
\end{itemize}
answers whether there exists an $R$-module $N' \in S'_{\beta, \gamma}$ satisfying the sentences $\vert \, x \eta = 0 \, / \, \beta |x \,  \vert =p^w$ and
\[\bigwedge_{i=1}^{t}\vert\phi_i/\psi_i\vert=p^{n_i}
\wedge\bigwedge_{j=1}^u \vert\phi_{2,j} / \psi_{2,j} \vert=1
\wedge\bigwedge_{i=t+1}^{t+s} \vert\phi_i / \psi_i \vert \geq p^{l_i}. \]
\end{enumerate}
\end{proposition}
\begin{proof}
We provide the proof of (a), and then we explain how it can be adapted to show (b) and (c).

Let $\Gamma$ be the set of functions $f:\{1,\ldots,t+s\}\rightarrow\{0,1,\infty\}$. Let $X$ be the set of pairs $(\Gamma',\delta)$ where $\Gamma'\subseteq \Gamma$ and $\delta:\Gamma'\rightarrow \N$ satisfy
\begin{enumerate}
\item $\sum_{f\in\Gamma'}\delta(f)=w$,
\item for $1\leq i\leq t$ and $f\in\Gamma'$, $f(i)\neq \infty$,
\item for $1\leq i\leq t$, $\sum_{f\in \Gamma', \, f(i)=1}\delta(f)=n_i$,
\item for $t+1\leq i \leq t+s$, either $\sum_{f\in \Gamma', \, f(i)=1}\delta(f)\geq l_i$ or there
exists $f\in\Gamma'$ such that $f(i) = \infty$.
\end{enumerate}

The first condition ensures that the set $X$ is finite and not empty.

Recall that, for each maximal ideal $\mfrak{m}$, $N_\gamma(\mfrak{m}) = R_{\mfrak{m}}/\gamma\mfrak{m}R_{\mfrak{m}}$. For each $f\in \Gamma$ we define $\Delta^f\in\mathbb{B}$, so that a maximal ideal $\mfrak{m}$ satisfies $\Delta^f$ if and only if, for $i = 1, \ldots, t+s$,
\begin{enumerate}
\item[($S1$)] $f(i)=0$ implies $\vert \, \phi_i \, / \, \psi_i \, (N_\gamma(\mfrak{m}))\, \vert=1$,
\item[($S2$)] $f(i)=1$ implies $\phi_i/\psi_i$ is an $N_\gamma(\mfrak{m})$-minimal pair,
\item[($S3$)] $f(i)=\infty$ implies $\vert \, \phi_i \, / \, \psi_i \, (N_\gamma(\mfrak{m})) \, \vert>1$ and $\phi_i/\psi_i$ is not an $N_\gamma(\mfrak{m})$-minimal pair,
\item[($S4$)] $\vert \, \phi_{2,j} \, / \, \psi_{2, j} \, (N_\gamma(\mfrak{m})) \, \vert = 1$ for $1 \leq j \leq u$.
\end{enumerate}

This can be done using Lemmas \ref{relmaxideal}, \ref{condpp}, (1) and (2), and \ref{condpp1}, (1) and (2). Note that the conditions in ($S$4) do not depend of $f$, and yet
are assumed to be part of $\Delta^f$.

We claim that there exists $N \in S_\gamma$ satisfying the sentences $\vert \, x=x \, / \, x\gamma=0 \, \vert =p^w$ and
\[\bigwedge_{i=1}^{t}\vert\phi_i/\psi_i\vert=p^{n_i}
\wedge\bigwedge_{j=1}^u \vert \phi_{2,j} / \psi_{2,j} \vert=1
\wedge\bigwedge_{i=t+1}^s\vert\phi_i/\psi_i\vert\geq p^{l_i} \] if and only if there exists $(\Gamma',\delta)\in X$ such that $(p,\delta(f),\Delta^f)\in \text{PP}_0 (R)$ for all $f\in \Gamma'$ and $w = \sum_{f \in \Gamma'} \delta (f)$. Since, by Lemma \ref{PPrecimpPP1rec}, $PP^\star (R)$ recursive implies $PP_0(R)$ recursive, this is enough to prove the proposition.

We first prove the forward direction. Suppose $N\in S_\gamma$ satisfies the required sentences. By definition $N= \oplus_{h=1}^mN_\gamma(\mfrak{m}_h)$ for some maximal ideals $\mfrak{m}_1,\ldots,\mfrak{m}_m$. Recall
that $N$ satisfies $\vert \, x=x \, / \, x\gamma=0 \, \vert =p^w$. On the other hand, for each $h=1, \ldots, m$, $\vert  N_\gamma (\mfrak{m}_h) \, / \, (x\gamma=0) (N_\gamma(\mfrak{m}_h))\vert =\vert R_{\mfrak{m}_h} / \mfrak{m}_h R_{\mfrak{m}_h} \vert = \vert R / \mfrak{m}_h \vert$, whence $\vert R/\mfrak{m}_h\vert$ is finite and indeed a power of $p$. For $1\leq h\leq m$, put $\vert R/\mfrak{m}_h\vert=p^{k_h}$. Then
$\sum_{h=1}^m k_h=w$.

Each $\mfrak{m}_h$ satisfies $\Delta^f$ for exactly one $f\in\Gamma$ because
$\vert \, \phi_{2j} \, / \, \psi_{2,j} (N_\gamma(\mfrak{m}_h)) \, \vert=1$ for $1\leq j\leq u$ and $\Delta^f$ simply specifies, for $1 \leq i \leq t+s$ and $\mfrak{m}$ a maximal ideal, whether $\vert \, \phi_i \, / \, \psi_i \, (N_\gamma(\mfrak{m})) \, \vert=1$, $\phi_i/\psi_i$ is an $N_\gamma(\mfrak{m})$-minimal pair or neither of these things is true. Let $\Gamma'$ be the set of $f\in\Gamma$ such that $\mfrak{m}_h$ satisfies $\Delta^f$ for some $1\leq h\leq m$. Since $\vert\phi_i / \psi_i \, (N) \vert$ is finite for $1\leq i\leq t$, for all $1\leq h\leq m$, $\vert \phi_i \, / \, \psi_i \, (N_\gamma(\mfrak{m}_h)) \vert$ is finite. Therefore $f\in\Gamma'$ implies $f(i)\neq \infty$ for $1\leq i\leq t$.

For each $f\in\Gamma'$, let $H_f$ be the set of $1\leq h\leq m$ such that $\mfrak{m}_h$ satisfies $\Delta^f$.
Define $\delta:\Gamma'\rightarrow \N$ by setting $\delta(f):=\sum_{h\in H_f}k_h$
for every $f\in\Gamma'$.

We show that $(\Gamma',\delta)\in X$. We have already seen that $w=\sum_{h=1}^mk_h$. Since for each $1\leq h\leq m$, $h\in H_f$ for exactly one $f\in\Gamma'$, $\sum_{f\in\Gamma'}\delta(f)=\sum_{h=1}^mk_h=w$. So $\delta$ satisfies condition (1).

We have already proved that $\Gamma'$ satisfies condition (2). So let us pass to (3).

Let $1\leq i\leq t$. Since $N$ satisfies $\vert\phi_i (N) /\psi_i(N)\vert=p^{n_i}$, Lemma \ref{finiteimpminpair} implies that, for all $1\leq h\leq m$, either $\phi_i/\psi_i$ is an $N_\gamma(\mfrak{m}_h)$-minimal pair or $\vert\phi_i \, / \, \psi_i \, (N_\gamma(\mfrak{m}_h))\vert =1$. Let $T_i$ be the set of $1\leq h\leq m$ such that $\phi_i/\psi_i$ is a $N_\gamma(\mfrak{m}_h)$-minimal pair. So $\sum_{h\in T_i} k_h=n_i$. Thus
\[\sum_{\substack{f\in\Gamma' \\ f(i)=1}}\delta(f)=\sum_{\substack{f\in\Gamma' \\ f(i)=1}}\sum_{h\in H_f}k_h=\sum_{h\in T_i} k_h=n_i.\] So $\delta$ satisfies condition (3).

Finally let us deal with (4). Let $t+1\leq i\leq t+s$. If there exists some $1\leq h\leq m$ such that $\phi_i/\psi_i$ is not an $N_\gamma(\mfrak{m}_h)$-minimal pair and $\vert \phi_i  \, / \, \psi_i \, (N_\gamma(\mfrak{m}_h))\vert\neq 1$ then there is an $f\in\Gamma'$ such that $f(i)=\infty$. In this case, $(\Gamma',\delta)$ satisfies condition (4). So suppose that for all $1\leq h\leq m$, $\phi_i/\psi_i$ is an $N_\gamma(\mfrak{m}_h)$-minimal pair or $\vert \phi_i \, / \, \psi_i \, (N_\gamma(\mfrak{m}_h))\vert= 1$. Let $T_i$ be the set of $1\leq h\leq m$ such that $\phi_i/\psi_i$ is a $N_\gamma(\mfrak{m}_h)$-minimal pair. So $\sum_{h\in T_i} k_h\geq l_i$. Thus
\[\sum_{\substack{f\in\Gamma' \\ f(i)=1}}\sum_{h\in H_f}k_h=\sum_{h\in T_i} k_h\geq l_i.\] So $\delta$ satisfies condition (4).

We now just need to confirm that $(p,\delta(f),\Delta^f)\in \text{PP}_0 (R)$ for all $f\in \Gamma'$. By definition $\delta(f)=\sum_{h\in H_f}k_h$. So $\delta(f)\in \text{Span}_{\N_0}\{k_h \st h\in H_f\}$. By definition for each $h\in H_f$, $\mfrak{m}_h$ satisfies $\Delta^f$. So $(p,\delta(f),\Delta^f)\in \text{PP}_0 (R)$ for all $f\in \Gamma'$.

We now prove the reverse direction. Suppose that there exists a pair $(\Gamma',\delta)\in X$ such that $(p,\delta(f),\Delta^f)\in \text{PP}_0 (R)$ for all $f\in \Gamma'$ and $w = \sum_{f \in \Gamma'} \delta (f)$. Using the definition of $\text{PP}_0 (R)$, for each $f\in \Gamma'$, pick maximal ideals $\mfrak{m}^f_1,\ldots \mfrak{m}^f_{m_f}$ such that $\mfrak{m}^f_h$ satisfies $\Delta^f$ for $1\leq h\leq m_f$ and $\delta(f)\in\text{Span}_{\N_0}\{k_h^f \st 1\leq h\leq m_f\}$ where $\vert R/\mfrak{m}^f_h\vert=p^{k_h^f}$. For $1 \leq h \leq m$, let $\lambda_h\in\N_0$ be such that $\delta(f)=\sum_{h=1}^{m_f}\lambda_hk_h^f$.

Let $N^f:=\oplus_{h=1}^{m_f}N_\gamma(\mfrak{m}_h^f)^{\lambda_h}$.

Note that $\vert x=x \, / \, x \gamma = 0 \, (N^f) \vert = p^{\delta(f)}$.

By definition of $\Delta^f$, for all $1 \leq j \leq u$ and $1\leq h\leq m_f$, $\vert \phi_{2,j} \, / \psi_{2,j}$ $(N_\gamma(\mfrak{m}_h^f))\vert=1$. Thus $\vert \phi_{2,j} \, / \, \psi_{2,j} \, (N^f) \vert = 1$ for all $1\leq j\leq u$.
Again by definition of $\Delta^f$, if $1\leq i\leq t+s$ and $f(i)=0$, then $\vert \phi_i /\psi_i \, (N^f)\vert=1$. If $1\leq i\leq t+s$ and $f(i)=1$, then $\phi_i/\psi_i$ is an $N_\gamma(\mfrak{m}_h)$-minimal pair for all $1\leq h\leq m_f$. Thus $\vert\phi_i / \psi_i \, (N^f)\vert=p^{\sum_{h=1}^{m_f}\lambda_hk_h^f}=p^{\delta(f)}$. Finally, if $f(i)=\infty$ then $\vert\phi_i  \, / \, \psi_i \, (N_\gamma(\mfrak{m}_h))\vert\neq 1$ and $\phi_i/\psi_i$ is not a $N_\gamma(\mfrak{m}_h)$-minimal pair for $1\leq h\leq m_f$. Thus $\vert\phi_i \, / \, \psi_i \, (N_\gamma(\mfrak{m}_h))\vert$ is infinite. Therefore, if $f(i)=\infty$ then $\vert\phi_i / \psi_i \, (N^f)\vert$ is infinite.

Let $N:=\oplus_{f\in\Gamma'} N^f$.

For $1\leq j\leq u$, $\vert\phi_{2,j} / \psi_{2,j} (N)\vert=1$ since $\vert\phi_{2,j} \, / \, \psi_{2,j} \, (N^f) \vert = 1$ for each $f\in\Gamma'$.

Let $1\leq i\leq t$. If $f\in\Gamma'$, $f(i)\neq \infty$. Since $(\Gamma',\delta)\in X$, $\sum_{i\in\Gamma' \ f(i)=1}\delta(f)=n_i$. Thus

\[\vert \phi_i /\psi_i(N)\vert=\prod_{f\in\Gamma'}\vert\phi_i \, / \, \psi_i \, (N^f)\vert = \] \[ \prod_{\substack{f\in\Gamma' \\ f(i)=1}}\vert\phi_i \, / \, \psi_i \, (N^f)\vert=\prod_{\substack{f\in\Gamma' \\ f(i)=1}} p^{\delta(f)}=p^{n_i}.\]

Let $t+1\leq i\leq t+s$. If $f(i)=\infty$ for some $f\in\Gamma'$ then $\vert\phi_i \, / \, \psi_i \, (N^f)\vert$ is infinite and hence $\vert\phi_i / \psi_i(N)\vert$ is infinite. So $\vert\phi_i / \psi_i(N)\vert\geq l_i$. So suppose that $f(i)\neq \infty$ for all $f\in\Gamma'$. That $\vert \phi_i / \psi_i(N)\vert \geq l_i$ now follows as in the previous paragraph.

Finally
$$ \vert x=x \, / \, x \gamma = 0 \, (N) \vert = \prod_{f \in \Gamma'}
\vert x=x \, / \, x \gamma = 0 \, (N^f) \vert =
$$
$$
= \prod_{f \in \Gamma'} p^{\delta (f)} = p^{\sum_{f \in \Gamma'} \delta (f)} = p^w.
$$
Thus we have shown that $N$ satisfies the required sentences.

This concludes the proof of (a).

For (b), when we define $\Delta^f$ we need to add the conditions $\beta\in M$ and $\eta\in M$ and use Lemmas \ref{condpp}, \ref{condpp1}, (1)-(2), with $\gamma=1$.

For (c),  when we define $\Delta^f$ we need to add the conditions $\beta \in M$, $\eta\in M$ and use Lemmas \ref{condpp}, \ref{condpp1}, (3)-(4).
\end{proof}

\section{The main theorem}\label{main}

\begin{theorem}\label{key2bis}
Let $R$ be an effectively given Pr\"ufer domain such that each localization of $R$ at a maximal ideal has dense value group. If both $\text{DPR}^\star (R)$ and $\text{PP}^\star (R)$ are recursive, then $T_R$ is decidable.
%
\end{theorem}

\begin{proof}
By Theorem \ref{reduction}, in order to prove that $T_R$ is decidable, it is enough that there is an algorithm which, given a conjunction $\sigma$ of invariants sentences
\begin{enumerate}
\item $| \phi_{1,i} / \psi_{1,i} | = H_i$, $1 \leq i \leq t$,
\item $| \phi_{2,j} / \psi_{2,j} | = 1$, $1 \leq j \leq u$,
\item $| \phi_{3,k} / \psi_{3,k} |\geq E_k$, $1 \leq k \leq s$,
\end{enumerate}
where $t$, $u$, $s$ are non negative integers, $H_i$ ($1 \leq i \leq s$) and $E_k$ ($1 \leq k \leq s$) are integers $> 1$ and all the involved pp-pairs have the form $x \eta = 0 \, / \, \beta |x$ or $x=x \, / \, x \gamma =0$
with $\beta, \eta, \gamma \in R$, answers whether there is some $R$-module $M$ satisfying $\sigma$.

For $\sigma$ a sentence as above, define the {\sl exponent} of $\sigma$ to be $\prod_{i=1}^tH_i$ if $(1)$ is non-empty and $1$ otherwise.

Our plan is to describe an algorithm for sentences of exponent $1$ and then explain how to algorithmically reduce to the exponent $1$ case.

\smallskip

\noindent
{\bf Case 1: the exponent of $\sigma$ is $1$.}

So (1) is empty. Then there exists an $R$-module satisfying $\sigma$ if and only if there exists an $R$-module satisfying
\[\sigma':= \bigwedge_{j=1}^u | \phi_{2,j} / \psi_{2,j} | =1 \wedge\bigwedge_{k=1}^s | \phi_{3,k} / \psi_{3,k} | >1 .\] This is because if $M$ satisfies $\sigma'$ then $M^{\aleph_0}$ satisfies $\sigma$. We may now proceed as in \cite[Theorem 7.1]{GLPT}.

\smallskip

\noindent
{\bf Case 2: the exponent of $\sigma$ is strictly greater than $1$.}

We now describe an algorithm which given a sentence $\sigma$ with exponent strictly greater than $1$ produces finitely many sentences $\sigma_1,\ldots,\sigma_l$ such that their exponents are strictly smaller than that of $\sigma$ and there exists an $R$-module satisfying $\sigma$ if and only if there exists an $R$-module satisfying one of the sentences $\sigma_1,\ldots,\sigma_l$.

Given a sentence $\sigma$, we can apply this algorithm finitely many times to produce sentences $\sigma_1,\ldots,\sigma_l$ with exponent $1$ such that $\sigma$ is true in some $R$-module if and only if one of the sentences $\sigma_1,\ldots,\sigma_l$ is true in some $R$-module. So we are done.

Let $p\in\mathbb{P}$ divide $H_1$ and $h_1\in \N$ be maximal such that $p^{h_1}|H_1$. We will deal with the cases when $\phi_{11}/\psi_{11}$ is of the form $x\eta=0 \, / \, \beta|x$ and $x=x \, / \, x\gamma=0$ separately.

\smallskip
\noindent
{\bf Subcase 2.1: $\phi_{11}/\psi_{11}$ is $x \eta=0 \, / \, \beta|x$.}

Let $\Omega_{\sigma}$ be the set of pairs $(f,g)$ of functions $f,g:\{1,\ldots,t+s\}\rightarrow \N_0\cup\{\infty\}$ such that
\begin{itemize}
\item $f(1)=0$ (respectively $g(1)=0$) implies $f(i)=0$ (respectively $g(i)=0$) for $1\leq i\leq t+s$,
\item $f(1)+g(1)=h_1$,
\item $p^{f(i)+g(i)}|H_i$ for $2\leq i\leq t$,
\item either $p^{f(t+k)}\leq E_k$ or $f(t+k)=\infty$ for $1\leq k \leq s$, and
\item either $p^{g(t+k)}\leq E_k$ or $g(t+k)=\infty$ for $1\leq k \leq s$.
\end{itemize}

For each $(f,g)\in \Omega_{\sigma}$, let $\sigma_f$ (respectively $\sigma_g$) be the conjunction of invariants sentences
\begin{enumerate}
\item $| \phi_{1,i} / \psi_{1,i} | = p^{f(i)}$ (respectively $=p^{g(i)}$) for $1 \leq i \leq t$,
\item $| \phi_{2,j} / \psi_{2,j} | = 1$ for $1 \leq j \leq u$,
\item $| \phi_{3,k} / \psi_{3,k} | = p^{f(t+k)}$ (respectively $=p^{g(t+k)}$) for $1 \leq k \leq s$ and $f(t+k)\neq \infty$ (respectively $g(t+k)\neq \infty$),
\item $| \phi_{3,k} / \psi_{3,k} | \geq p^{\left \lceil{\log_p E_k}\right \rceil }$ for $1 \leq k \leq s$ and $f(t+k)=\infty$ (respectively $g(t+k)=\infty$).
\end{enumerate}

Here, for $r$ a real number, $\lceil r \rceil$ denotes the minimal integer greater than or equal to $r$ .
For each $(f,g)\in\Omega_\sigma$, let $\sigma_{(f,g)}'$ be the conjunction of invariants sentences

\begin{enumerate}
\item $| \phi_{1,i} / \psi_{1,i} | =H_i/p^{f(i)+g(i)}$ for $1\leq i\leq t$,
\item $| \phi_{2,j} / \psi_{2,j} | = 1$ for $1 \leq j \leq u$,
\item  $| \phi_{3, k} / \psi_{3, k} | \geq \left \lceil{(E_k / p^{f(t+k)+g(t+k)})}\right \rceil$ for $1\leq k \leq s$, $f(t+k)\neq \infty$ and $g(t+k)\neq \infty$.
\end{enumerate}

Note that, for each $(f,g)\in\Omega_\sigma$, the exponent of $\sigma_{(f,g)}'$ is strictly less that the exponent of $\sigma$ since $p^{f(1)+g(1)}=p^{h_1}>1$.

If $p^{f(t+k)+g(t+k)}\geq E_k$ then $\left \lceil{(E_k/p^{f(t+k)+g(t+k)})}\right \rceil=1$. For that value of $k$, $| \phi_{3,k} / \psi_{3,k} | \geq \left \lceil{(E_k/p^{f(t+k)+g(t+k)})}\right \rceil$ is satisfied by all $R$-modules and so this condition may be removed.

Now we claim that  {\sl there exists an $R$-module $M$ satisfying $\sigma$ if and only if there exists $(f,g)\in\Omega_{\sigma}$, $N_f\in S'_{\beta,\eta}$ satisfying $\sigma_f$, $N_g\in T_{\beta,\eta}$ satisfying $\sigma_g$ and an $R$-module $M'$ satisfying $\sigma_{(f,g)}'$.}

The reverse direction follows directly from the definitions of $\sigma_f$, $\sigma_g$ and $\sigma_{(f,g)}'$.

So assume that there is an $R$-module $M$ satisfying $\sigma$. We may suppose that $M$ is a finite direct sum of indecomposable pure injective modules $\bigoplus_{h=1}^mN_h$. So $\prod_{h=1}^m |x\eta=0 \, / \, \beta|x \, (N_h)|=H_1$. In particular, $|x \eta=0 \, / \, \beta|x \, (N_h)|$ is finite for each $1\leq h\leq m$. For each $1\leq h\leq m$, $|x\eta=0 \, / \, \beta|x \, (N_h)|$ is either $1$ or $q^l$ for some prime $q$. If $|x\eta =0 \, / \, \beta|x \, (N_h)|=q^l$ for some $l\in\N$ then, by Lemma \ref{finiteimpminpair}, $x\eta=0 \, / \, \beta|x$ is an $N_h$-minimal pair. So $N_h$ is either the pure injective hull of $R/\mfrak{m}_{h}$ or $\mfrak{m}_{h}R_{\mfrak{m}_{h}}/\beta\eta R_{\mfrak{m}_{h}}$ for some maximal ideal $\mfrak{m}_h$ with $\beta,\eta\in \mfrak{m}_h\backslash\{0\}$ and such that $|R/\mfrak{m}_h|=q^l$.

Let $M'$ be the direct sum of the modules $N_h$ for $1\leq h\leq m$ such that $p$ does not divide $|x\eta =0 \, / \, \beta|x \, (N_h)|$. Let $L$ (respectively $N$) be the direct sum of the modules $N_h$ for $1\leq h\leq m$ such that $p$ divides $|x\eta=0 \, / \, \beta|x \, (N_h)|$ and $N_h$ the pure injective hull of $R/\mfrak{m}_{h}$ with $\beta,\eta \in\mfrak{m}_h$ (respectively $\mfrak{m}_{h}R_{\mfrak{m}_{h}}/\beta \eta R_{\mfrak{m}_{h}}$ with $\beta,\eta\in\mfrak{m}_h\backslash\{0\}$).

Note that $|x\eta=0 \, / \, \beta|x(L\oplus N)|=p^{h_1}$ and for any pp-pair $\phi/\psi$, $|\phi / \psi(L\oplus N)|$ is either $1$, a power of $p$ or infinite. Moreover, if $|x\eta=0 \, / \, \beta|x \, (L)|=1$ (respectively $|x\eta=0 \, / \, \beta|x \, (N)|=1$) then $L=0$ (respectively $N=0$).

Define $(f,g)\in \Omega_\sigma$ by setting \[f(i)= \log_p|\phi_{1i}/\psi_{1i}(N)| \text{ and } g(i)=\log_p|\phi_{1i}/\psi_{1i}(L)|\] for $1\leq i\leq t$ and, for $1 \leq k \leq s$,
\[f(t+k)=\left\{
           \begin{array}{ll}
             \log_p|\phi_{3,k}/\psi_{3,k}(N)|, & \hbox{if $|\phi_{3,k}/\psi_{3,k}(N)|\leq E_k$,} \\
             \infty, & \hbox{otherwise}
           \end{array}
         \right.
\] and
\[g(t+k)=\left\{
           \begin{array}{ll}
             \log_p|\phi_{3,k}/\psi_{3,k}(L)|, & \hbox{if $|\phi_{3,k}/\psi_{3,k}(L)|\leq E_k$,} \\
             \infty, & \hbox{otherwise.}
           \end{array}
         \right.
\]

Then $N$ satisfies $\sigma_f$ and $L$ satisfies $\sigma_g$. That $(f,g)\in \Omega_{\sigma}$ follows from the definition of $N$ and $L$ and the above discussion.  We now just need to check that $M'$ satisfies $\sigma_{(f,g)}'$.
For $1\leq i\leq t$, \[\vert \phi_{1,i}/\psi_{1,i}(N) \vert\cdot \vert\phi_{1,i}/\psi_{1,i}(L) \vert\cdot \vert\phi_{1,i}/\psi_{1,i}(M')\vert = \] \[= p^{f(i)+g(i)}\cdot\vert\phi_{1i}/\psi_{1i}(M')\vert=H_i.\]

So $M'$ satisfies the invariants sentences in $(1)$ of the definition of $\sigma_{(f,g)}'$.

Finally, suppose $1\leq k \leq s$, $f(t+k)\neq \infty$ and $g(t+k)\neq \infty$. Then
\[p^{f(t+k)+g(t+k)}\cdot\vert\phi_{3,k}/\psi_{3,k}(M')\vert = \] \[ = \vert \phi_{3,k}/\psi_{3,k}(N) \vert\cdot \vert\phi_{3,k}/\psi_{3,k}(L) \vert\cdot \vert\phi_{3,k}/\psi_{3,k}(M')\vert\geq E_k.\]
Therefore $\vert\phi_{3,k}/\psi_{3,k}(M')\vert\geq \left \lceil{E_k/p^{f(t+k)+g(t+k)}}\right \rceil$. So $M'$ satisfies the invariants sentences in $(2)$ of the definition of $\sigma_{(f,g)}'$.

Thus we have proved the claim.

By definition, $N$ is elementary equivalent to a module in $S'_{\beta,\eta}$ and $L$ is elementary equivalent to a module in $T_{\beta,\eta}$. If $f(1)=0$ (respectively $g(1)=0$) then $\sigma_f$ (respectively $\sigma_g$) holds for the zero module. If $f(1)\neq 0$ then, by Proposition \ref{prepp}(c), there is an algorithm which given the sentence $\sigma_f$ answers whether there exists $N\in S'_{\beta,\eta}$ satisfying $\sigma_f$. Similarly, if $g(1)\neq 0$ then, by Proposition \ref{prepp}(b), there is an algorithm which given the sentence $\sigma_g$ answers whether there exists $L\in T_{\beta,\eta}$ satisfying $\sigma_g$.

Now $\sigma$ is true in some $R$-module if and only if there exists $(f,g)\in \Omega_{\sigma}$ such that $\sigma_f$ is true in some module in $S'_{\beta,\eta}$ and $\sigma_g$ is true in some module in $T_{\beta,\gamma}$ and $\sigma_{(f,g)}'$ is true in some $R$-module.

\smallskip
\noindent
{\bf Subcase 2.2: $\phi_{11}/\psi_{11}$ is $x=x \, / \, x\gamma=0$.}

Let $\Omega_{\sigma}$ be the set of functions $f:\{1,\ldots,t+s\}\rightarrow \N_0\cup\{\infty\}$ such that
\begin{itemize}
\item $f(1)=h_1$,
\item $p^{f(i)}|H_i$ for $2\leq i\leq t$,
\item either $p^{f(t+k)}\leq E_k$ or $f(t+k)=\infty$ for $1\leq k \leq s$.
\end{itemize}

For each $f\in \Omega_{\sigma}$, let $\sigma_f$ be the conjunction of invariants sentences
\begin{enumerate}
\item $| \phi_{1,i} / \psi_{1,i} | = p^{f(i)}$, for $1 \leq i \leq t$,
\item $| \phi_{2,j} / \psi_{2,j} | = 1$ for $1 \leq j \leq u$,
\item $| \phi_{3,k} / \psi_{3,k} | = p^{f(t+k)}$, if $1 \leq k \leq s$ and $f(t+k)\neq \infty$,
\item $| \phi_{3,k} / \psi_{3,k} | \geq p^{\left \lceil{\log_p E_k}\right \rceil }$, if $1 \leq k \leq s$ and $f(t+k)=\infty$.
\end{enumerate}

For each $f\in\Omega_\sigma$, let $\sigma_{f}'$ be the conjunction of invariants sentences

\begin{enumerate}
\item $| \phi_{1,i} / \psi_{1,i} | =H_i/p^{f(i)}$,
\item $| \phi_{2,j} / \psi_{2,j} | = 1$ for $1 \leq j \leq u$,
\item  $| \phi_{3,k} / \psi_{3,k} | \geq \left \lceil{(E_k / p^{f(t+k)})}\right \rceil$ for $f(t+k)\neq \infty$.
\end{enumerate}

Note that if $p^{f(t+k)} \geq E_k$ then $\left \lceil{(E_k / p^{f(t+k)})}\right \rceil=1$ and so the last condition for that value of $k$ is satisfied by all $R$-modules.

We claim that {\sl there exists an $R$-module $M$ satisfying $\sigma$ if and only if there exist $f\in\Omega_{\sigma}$, $N_f\in S_{\gamma}$ satisfying $\sigma_f$, and an $R$-module $M'$ satisfying $\sigma_f'$}.

The proof of this claim is as the previous one except we use the fact that the only indecomposable pure injective $R$-modules $N$ such that $x=x \, / \, x\gamma=0$ is an $N$-minimal pair are the pure injective hulls of $N_\gamma(\mfrak{m})$.

We can now use Proposition \ref{prepp}(a) instead of \ref{prepp}(b) and (c) to get the required algorithm.
\end{proof}

Recall that, over a B\'ezout domain $R$,
if $\text{PP}(R)$ is recursive then $\text{PP}^\star (R)$ is recursive and if
$\text{DPR}(R)$ is recursive then $\text{DPR}^\star (R)$ is recursive (see $\S$\ref{intro}). Moreover, \cite[6.4]{GLPT}, if $T_R$ is decidable then $\text{DPR}(R)$ is recursive.

When $R$ is a B\'ezout domain, we obtain the following theorem as a corollary to Theorem \ref{key2bis} and Theorem \ref{key1bis}.

\begin{theorem}\label{key33}
Let $R$ be an effectively given B\'ezout domain such that each localization of $R$ at a maximal ideal has dense value group. Then $T_R$ is decidable if and only if both $\text{DPR} (R)$ and $\text{PP}(R)$ are recursive.
\end{theorem}

The following remark follows directly from the definition of $\text{PP}(R)$.

\begin{remark}\label{constantres}
Let $q\in \mathbb{P}$ and $t\in\N$. Let $R$ be a Pr\"{u}fer domain such that for all maximal ideals $\mfrak{m}$, $|R / \mfrak{m}|=q^t$. Then $(p,n,c,d)\in \text{PP}(R)$ if and only if $p=q$, $t$ divides $n$ and there exists some maximal ideal $\mfrak{m}$ such that $c\in\mfrak{m}$ and $d\notin\mfrak{m}$.
\end{remark}
%

When the Krull dimension of $R$ is 1, we can say a bit more.

\begin{proposition}\label{kdim}
Let $R$ be an effectively given B\'ezout domain of Krull dimension 1 such that, when $\mfrak{m}$ ranges over the maximal ideals of $R$, either all the residue fields $R/ \mfrak{m}$ are infinite, or all the residue fields $R / \mfrak{m}$ are of the same finite cardinality and the localizations $R_\mfrak{m}$ have dense value group. Then $T_R$ is decidable.
\end{proposition}

\begin{proof}
Note that the case of infinite residue fields is already treated in \cite[Corollary 6.7]{GLPT}. So we focus on the second case, when the residue fields of $R$ have a common finite size, $q^t$ say, where $q$ is a prime and $t$ is a positive integer. It suffices to prove that then both $\text{PP}(R)$ and $\text{DPR}(R)$ are recursive. It is shown in \cite[Lemma 3.3]{LPT} that, when $R$ is an effectively given B\'ezout domain with Krull dimension 1, the prime radical relation is recursive.

The argument for $\text{DPR}(R)$ is just the same as in the proof of \cite[Corollary 6.7]{GLPT}. So we just need to show that $\text{PP}(R)$ is recursive.

First we show that $\text{Jac}(R)$, the Jacobson radical of $R$, is a recursive subset of $R$. If $\text{Jac}(R)=0$ then it is finite and hence recursive. Suppose $\text{Jac}(R)\neq 0$. Fix $a\in \text{Jac}(R)$ non-zero. Then $a$ is contained in all maximal ideals and all non-zero prime ideals are maximal. Therefore $\rad(aR)=\text{Jac}(R)$. Since the prime radical relation is recursive, $\text{Jac}(R)$ is a recursive subset of $R$.

By \ref{constantres}, if all residue fields of $R$ are of the same finite size then $(p, n, c, d) \in \text{PP}(R)$ if and only if $t$ divides $n$ and there exists some maximal ideal $\mfrak{m}$ such that $c\in\mfrak{m}$ and $d\notin \mfrak{m}$. Therefore, it is enough to show that the set of pairs $(c,d)\in R^2$ such that there exists some maximal ideal $\mfrak{m}$ with $c\in\mfrak{m}$ and $d\notin \mfrak{m}$ is recursive.

Suppose $c\neq 0$. Then $c\in\mfrak{m}$ and $d\notin \mfrak{m}$ implies that $d\notin \rad(cR)$. Conversely, if $d\notin \rad(cR)$ then there exists some prime ideal $\mfrak{m}$ such that $c\in\mfrak{m}$ and $d\notin \mfrak{m}$. Since $c\neq 0$, $\mfrak{m}$ is non-zero and hence, since $R$ is a domain with Krull dimension $1$, $\mfrak{m}$ is maximal.

Now, if $c=0$ then $c$ is a member of all maximal ideals. So there exists a maximal ideal $\mfrak{m}$ such that $c\in\mfrak{m}$ and $d\notin \mfrak{m}$ if and only if $d\notin J(R)$.

We have shown that the set of $(p,n,c,d)\in\text{PP}(R)$ with $c\neq 0$ is recursive and that the set of $(p,n,0,d)\in\text{PP}(R)$ is recursive. So $\text{PP}(R)$ is recursive.
\end{proof}

\end{document}

%% file: macros.tex
\usepackage{amsthm}
\usepackage{amssymb}
\usepackage{amsmath}
\usepackage[shortlabels]{enumitem}

\input xy
\xyoption{all} 

\newtheorem{theorem}{Theorem}
\newtheorem*{theorem*}{Theorem}

\newtheorem{definition}[theorem]{Definition}
\newtheorem*{definition*}{Definition}
\newtheorem{lemma}[theorem]{Lemma}

\newtheorem{proposition}[theorem]{Proposition}
\newtheorem{remark}[theorem]{Remark}

\newtheorem*{cor*}{Corollary}
\newtheorem{fact}[theorem]{Fact}

\newtheorem*{conjecture*}{Conjecture}
\numberwithin{theorem}{section}

\renewcommand\iff{%
\ifmmode\text{ if and only if }%
\else if and only if \fi}

\renewcommand{\and}{\wedge}

\renewcommand{\phi}{\varphi}

\newcommand{\Span}{\text{Span}}

\newcommand{\rad}{\textnormal{rad}}

\renewcommand{\P}{\mathbb{P}}

\newcommand{\Ass}{\textnormal{Ass}}
\newcommand{\Div}{\textnormal{Div}}

\newcommand{\ann}{\textnormal{ann}}

\newcommand{\Zg}{\textnormal{Zg}}

\newcommand{\mfrak}[1]{\mathfrak{#1}}
\newcommand{\st}{\ \vert \ }

\newcommand{\N}{\mathbb{N}}